\documentclass{amsart}
\usepackage{graphicx,color,pdfsync}
\usepackage{amssymb,enumerate}
\usepackage[bookmarksnumbered,colorlinks,plainpages,backref]{hyperref}
\usepackage{ulem}

\newtheorem{thm}{Theorem}[section]
\newtheorem{cor}[thm]{Corollary}
\newtheorem{lem}[thm]{Lemma}
\newtheorem{prop}[thm]{Proposition}
\theoremstyle{definition}
\newtheorem{defn}[thm]{Definition}
\theoremstyle{remark}
\newtheorem{rem}[thm]{Remark}
\numberwithin{equation}{section}
\newcommand{\norm}[1]{\left\Vert#1\right\Vert}
\newcommand{\abs}[1]{\left\vert#1\right\vert}
\newcommand{\set}[1]{\left\{#1\right\}}

\newcommand{\eps}{\varepsilon}

\newcommand{\dbar}{\bar\partial}
\newcommand{\ddbar}{\partial\bar\partial}
\newcommand{\atopp}[2]{\genfrac{}{}{0pt}{}{#1}{#2}}
\DeclareMathOperator{\dom}{Dom}

\DeclareMathOperator{\re}{Re}
\DeclareMathOperator{\im}{Im}
\DeclareMathOperator{\supp}{supp}
\DeclareMathOperator{\dist}{dist}

\renewcommand{\H}{\mathcal H}
\newcommand{\LL}{\bar L}
\newcommand{\Om}{\Omega}

\newcommand{\dbars}{\bar\partial^*}

\newcommand{\dbarb}{\bar\partial_b}

\DeclareMathOperator{\Dom}{Dom}
\DeclareMathOperator{\Ran}{Range}

\newcommand{\bd}{\partial}

\newcommand{\p}{\partial}

\newcommand{\z}{\bar z}

\newcommand{\R}{\mathbb R}

\newcommand{\N}{\mathbb N}
\newcommand{\C}{\mathbb C}
\newcommand{\opL}{\mathcal{L}}

\newcommand{\ep}{\epsilon}
\newcommand{\I}{\mathcal{I}}

\newcommand{\vp}{\varphi}

\begin{document}

\title[Strong Closed Range Estimates]{Strong Closed Range Estimates: Necessary Conditions and Applications}%
\author{Phillip S. Harrington and Andrew Raich}%
\address{SCEN 309, 1 University of Arkansas, Fayetteville, AR 72701}%
\email{psharrin@uark.edu, araich@uark.edu}%


\begin{abstract}
The $L^2$ theory of the $\dbar$ operator on domains in $\C^n$ is predicated on establishing a good basic estimate. Typically, one proves not a single basic estimate
but a family of basic estimates that we call a family of strong closed range estimates.
Using this family of estimates on $(0,q)$-forms as our starting point, we establish necessary geometric
and potential theoretic conditions.

The paper concludes with several applications. We investigate the consequences for compactness estimates for the $\dbar$-Neumann problem, and we also establish a generalization of Kohn's weighted theory via
elliptic regularization. Since our domains are not necessarily pseudoconvex, we must take extra care with the regularization.
\end{abstract}

\maketitle

\section{Introduction}\label{sec:Introduction}
Since H\"ormander's pivotal work on the $L^2$-theory of the $\dbar$-problem \cite{Hor65}, there has been a tremendous effort to characterize the regularity properties of the
$\dbar$-Neumann operator in terms of estimates, geometry, and potential theory. It has been known since the 1960s that pseudoconvexity is both necessary and sufficient for the range
of the $\dbar$-operator to be closed \emph{at every form level} $1 \leq q \leq n$ and for the absence of nontrivial harmonic forms at every form level $1\leq q\leq n$ \cite{Hor65,AnGr62}. The primary tool that analysts use to prove
closed range and other related properties is to establish an appropriate basic estimate or family of basic estimates. For example, the basic estimate
\begin{equation}\label{eqn:simple basic estimate}
c_\vp \| f\|_{L^2(\Om,\vp)}^2 \leq \big(\|\dbar f\|_{L^2(\Om,\vp)}^2 + \|\dbars_\vp f\|_{L^2(\Om,\vp)}^2\big) + C_\vp \|f\|_{W^{-1}(\Om,\vp)}^2
\end{equation}
for all $f \in L^2_{0,q}(\Om,\vp)\cap\Dom(\dbar)\cap\Dom(\dbars_\vp)$ suffices to show that the space of harmonic forms is finite dimensional and $\dbar$ has closed range in $L^2_{0,q}(\Omega)$ and $L^2_{0,q+1}(\Omega)$. It turns out that in every case where \eqref{eqn:simple basic estimate} is known to hold, we can actually prove a family of estimates -- namely, instead of \eqref{eqn:simple basic estimate} holding
for a single function $\vp$, we have \eqref{eqn:simple basic estimate} for every $\vp = t\phi$ where $t$ is sufficiently large
and $\phi$ is some fixed function and (typically) $c_{t\phi} = tC_q$ and $C_{t\phi} \leq O(t^2)$. This is an example of what we call a family of strong closed range estimates.

In this paper, we take strong closed range estimates as our starting point and explore the consequences for a domain admitting such a family. Our main result establishes a certain quantitative condition
on the number of nonnegative/nonpositive eigenvalues of the Levi form (a geometric condition) as well as the number of positive/negative eigenvalues of the complex Hessian of the weight
function restricted to $T^{1,0}_p(\p\Om)\times T^{0,1}_p(\p\Om)$ (a potential theoretic condition). Our main result has an application to compactness estimates for the
$\dbar$-problem, and the existence of a family of strong closed ranged estimates allows us to establish a generalization of Kohn's weighted theory for solving the weighted $\dbar$-Neumann operator
in $L^2$ Sobolev spaces. To prove this extension on non-pseudoconvex domains, we need to account for the possibility of non-trivial harmonic forms and the boundary condition induced by $\Dom(\dbars)$ -- typically, the weighted
theory has only one of these issues (non-trivial harmonic forms for $\dbarb$ on CR manifolds without boundary, and the boundary condition induced by $\Dom(\dbars)$ for $\dbar$ on pseudoconvex domains in $\mathbb{C}^n$), and we are particularly careful to avoid pitfalls that sometimes appear in the literature.

Surprisingly, since H\"ormander's work, the only results regarding closed range of $\dbar$ for $(0,q)$-forms on not-necessarily pseudoconvex domains have been to establish
\emph{sufficient} conditions and until the present work, none have attempted to find \emph{necessary} conditions. For example there are
several papers on the annulus or annular regions between two pseudoconvex domains \cite{Sha85a,Sha10,Hor04}, and we have investigated very general sufficient conditions
for both $\dbar$ and $\dbarb$ to have closed range. In fact, in the
language of this paper, we prove the existence of a family of strong closed range estimates and also establish a generalization of Kohn's weighted theory \cite{HaRa11,HaRa15,HaRa17z,HaRa18,HaRa19,HaRa20}.

The format of the paper is the following: we state the Main Results at the end of this section, define our notation and operators in Section \ref{sec:notation}, prove the main theorem
regarding strong closed range estimates in Section \ref{sec:SCRE}, and present the applications in Section \ref{sec:apps}.

\subsection{Statements of the Main Results}
\begin{thm}
\label{thm:main_theorem}
  Let $\Omega\subset\mathbb{C}^n$ be a bounded domain with $C^2$ boundary admitting a family of strong closed range estimates for some $1\leq q\leq n-1$ and some weight function $\varphi\in C^2(\overline\Omega)$, as in Definition \ref{defn:strong_closed_range} below.  Then for each connected component $S$ of $\partial\Omega$, one of the following two cases holds:
  \begin{enumerate}
    \item For every $z\in S$, the Levi form for $\partial\Omega$ has at least $n-q$ nonnegative eigenvalues and the restriction of $i\ddbar\varphi$ to $T^{1,0}_z(\partial\Omega)\times T^{0,1}_z(\partial\Omega)$ has at least $n-q$ positive eigenvalues bounded below by $\frac{C_q}{q}$.
    \item For every $z\in S$, the Levi form for $\partial\Omega$ has at least $q+1$ nonpositive eigenvalues and the restriction of $i\ddbar\varphi$ to $T^{1,0}_z(\partial\Omega)\times T^{0,1}_z(\partial\Omega)$ has at least $q+1$ negative eigenvalues bounded above by $-\frac{C_q}{n-q-1}$.
  \end{enumerate}
  If $\Omega$ admits a family of strong closed range estimates near some $p\in\partial\Omega$, then either $(1)$ or $(2)$ holds at $z=p$.
\end{thm}

\begin{rem}
  We have stated our result in a form which avoids technical details about the relationship between the Levi form and the complex hessian of $\varphi$, but we actually prove a much stronger statement.  Let $\rho$ be a defining function for $\Omega$ normalized so that $|\nabla\rho|=1$ on $\partial\Omega$.  For a constant $s\geq 0$ and $z\in\partial\Omega$, let $\mathcal{L}_s(z)$ denote the linear combination $i\ddbar\rho+s i\ddbar\varphi$ restricted to $T^{1,0}_z(\partial\Omega)\times T^{0,1}_z(\partial\Omega)$, and let $\{\lambda_1^s(z),\ldots,\lambda_{n-1}^s(z)\}$ denote the eigenvalues of $\mathcal{L}_s(z)$ in nondecreasing order.  Then for each connected component $S$ of $\partial\Omega$, we either have $\lambda_q^s(z)\geq\frac{s C_q}{q}$ for all $z\in S$ and $s\geq 0$ or we have $\lambda_{q+1}^s(z)\leq-\frac{s C_q}{n-q-1}$  for all $z\in S$ and $s\geq 0$.  If, for example, we are in the first case but $\lambda_q^0(z)=0$ for some $z$, then this more refined result can be used to deduce information about the restriction of $i\ddbar\varphi$ to the kernel of the Levi form at $z$.
\end{rem}

\begin{cor}
\label{cor:pseudoconvexity}
  Let $\Omega\subset\mathbb{C}^n$ be a bounded domain with $C^2$ boundary, and write $\Omega=\Omega_0\backslash\bigcup_{j\in J}\overline\Omega_j$ where   $\Omega_0$ is a bounded domain with connected $C^2$ boundary and $\{\Omega_j\}_{j\in J}$ is a collection of domains with connected $C^2$ boundaries that are relatively compact in $\Omega$ such that $\{\overline\Omega_j\}_{j\in J}$ is disjoint.  If $\Omega$ admits a family of strong closed range estimates for $q=1$, then $\Omega_0$ is pseudoconvex and the restriction of $i\ddbar\varphi$ to $T^{1,0}_p(\partial\Omega_0)\times T^{0,1}_p(\partial\Omega_0)$ is positive definite.  If $\Omega$ admits a family of strong closed range estimates for $q=n-2$, then each $\Omega_j$ is pseudoconvex and the restriction of $i\ddbar\varphi$ to $T^{1,0}_p(\partial\Omega_j)\times T^{0,1}_p(\partial\Omega_j)$ is negative definite.  If $\Omega$ admits a family of strong closed range estimates for $q=n-1$, then $\Omega=\Omega_0$.
\end{cor}

We will see that Definition \ref{defn:strong_closed_range} involves a family of smooth, compactly supported functions $\{\chi_t\}$ satisfying the growth condition $\lim_{t\rightarrow\infty}\frac{\norm{\chi}^2_{C^1(\Omega)}}{t^3}=0$.  This may seem to be a technical convenience, but in fact this distinguishes strong closed range estimates, which require a non-trivial weight function $\varphi$, from stronger families of estimates, which hold with no weight function.  For example, we have
\begin{prop}
\label{prop:subelliptic}
  Let $\Omega\subset\mathbb{C}^n$ be a domain with $C^2$ boundary such that for some $1\leq q\leq n-1$ and some $\eta>0$, $\Omega$ admits a subelliptic estimate of the form
  \[
  \norm{u}^2_{W^{\eta}(\Omega)}\leq C\left(\norm{\dbar u}^2_{L^2(\Omega)}+\norm{\dbar^* u}^2_{L^2(\Omega)}\right)
\]
for all $u\in L^2_{(0,q)}(\Omega)\cap\dom\dbar\cap\dom\dbar^*$.  Then $\Omega$ admits a family of estimates of the form \eqref{eq:strong_closed_range} for $\varphi=0$ and a family of smooth, compactly supported functions $\{\chi_t\}$ such that
\[
  0<\limsup_{t\rightarrow\infty}\frac{\norm{\chi_t}^2_{C^1(\Omega)}}{t^{(1+\eta)/\eta}}<\infty.
\]
\end{prop}
On strictly pseudoconvex domains, we have subelliptic estimates for $\eta=\frac{1}{2}$, and hence the family of cutoff functions $\{\chi_t\}$ given by Proposition \ref{prop:subelliptic} satisfies
\[
  0<\limsup_{t\rightarrow\infty}\frac{\norm{\chi_t}^2_{C^1(\Omega)}}{t^3}<\infty.
\]
This is the sense in which the growth condition in Definition \ref{defn:strong_closed_range} is sharp: if we relax this growth condition, then we have a large class of examples admitting a family of estimates of the form \eqref{eq:strong_closed_range} such that the conclusions of Theorem \ref{thm:main_theorem} do not hold.

As an immediate consequence of our main theorem, we have the following application to the compactness theory for the $\dbar$-Neumann problem.
\begin{thm}
\label{thm:compactness}
  Let $\Omega\subset\mathbb{C}^n$ be a domain with $C^2$ boundary.  Suppose that  $\Omega$ admits a family of compactness estimates for some $1\leq q\leq n-1$, as in Definition \ref{defn:compactness_estimate} below.  If $C_\eps$ denotes the constant in \eqref{eq:compactness}, then
    \begin{equation}
    \label{eq:quantitative_compactness}
      \limsup_{\eps\rightarrow 0^+}\eps^2 C_\eps>0.
    \end{equation}
\end{thm}

Our second and final application is to establish the weighted $L^2$-theory for the $\dbar$-problem in the presence of a family of strong closed range estimates.
\begin{thm}\label{thm:regularity of B and sol'n ops}
Let $\Om\subset\C^n$ be a smooth domain which admits the family of strong closed range estimates \eqref{eq:temp_strong_closed_range_2}
for some smooth function $\varphi$. Then for every $k\geq 1$ there exists $T_k$ so that if $t \geq T_k$, the following operators are continuous for all
$0 \leq s \leq k$:
\begin{enumerate}[i.]
\item The $\dbar$-Neumann operator
\[
N^q_{t\varphi}: L^{2,s}_{0,q}(\Om,t\varphi) \to L^{2,s}_{0,q}(\Om,t\varphi);
\]
\item The weighted canonical solution operators for $\dbar$ and $\dbars_{t\varphi}$:
\begin{align*}
\dbars_{t\varphi} N^q_{t\varphi} &: L^{2,s}_{0,q}(\Om,t\varphi) \to L^{2,s}_{0,q-1}(\Om,t\varphi); \\
\dbar N^q_{t\varphi} &: L^{2,s}_{0,q}(\Om,t\varphi) \to L^{2,s}_{0,q-1}(\Om,t\varphi);\\
N^q_{t\varphi}\dbars_{t\varphi} &: L^{2,s}_{0,q-1}(\Om,t\varphi) \to L^{2,s}_{0,q}(\Om,t\varphi);\\
N^q_{t\varphi}\dbar &: L^{2,s}_{0,q-1}(\Om,t\varphi) \to W^{s+1}_{0,q}(\Om,t\varphi);
\end{align*}
\item The projection operators:
\begin{align*}
  \dbar\dbars_{t\varphi} N^q_{t\varphi}&: L^{2,s}_{0,q}(\Om,t\varphi) \to L^{2,s}_{0,q}(\Om,t\varphi),\\
  \dbar N^q_{t\varphi} \dbars_{t\varphi}& : L^{2,s}_{0,q-1}(\Om,t\varphi) \to L^{2,s}_{0,q-1}(\Om,t\varphi),\\
  \dbars_{t\varphi}\dbar N^q_{t\varphi}&: L^{2,s}_{0,q}(\Om,t\varphi) \to L^{2,s}_{0,q}(\Om,t\varphi),\\
  \dbars_{t\varphi} N^q_{t\varphi}\dbar& : L^{2,s}_{0,q-1}(\Om,t\varphi)\to L^{2,s}_{0,q-1}(\Om,t\varphi).
\end{align*}

\item The harmonic projection $H^q_{t\varphi}:L^{2,s}_{0,q}(\Om,t\varphi) \to L^{2,s}_{0,q}(\Om,t\varphi)$.
\end{enumerate}
\end{thm}

\begin{rem}
  Note that we also obtain estimates for the weighted Bergman projections $P^q_{t\varphi}=I-\dbars_{t\varphi}\dbar N^q_{t\varphi}$ and $P^{q-1}_{t\varphi}=I-\dbars_{t\varphi} N^q_{t\varphi}\dbar$, as well as the combined projections $P^q_{t\varphi}+H^q_{t\varphi}=I-\dbar\dbars_{t\varphi} N^q_{t\varphi}$ and $P^{q+1}_{t\varphi}+H^{q+1}_{t\varphi}=I-\dbar N^q_{t\varphi} \dbars_{t\varphi}$.
\end{rem}

\begin{rem}
  We can also obtain estimates for the projection $N^q_{t\varphi}\dbar\dbars_{t\varphi}$ (resp.\\ $N^q_{t\varphi}\dbars_{t\varphi}\dbar$), but note that this is equal to the restriction of $\dbar\dbars_{t\varphi} N^q_{t\varphi}$ (resp., $\dbars_{t\varphi}\dbar N^q_{t\varphi}$) to the space of forms $u\in\dom\dbars_{t\varphi}$ such that $\dbars_{t\varphi} u\in\dom\dbar$ (resp., $u\in\dom\dbar$ such that $\dbar u\in\dom\dbars_{t\varphi}$).  The argument in \cite[(18)-(20)]{HaRa11} proves this for the complex Green operator, but the argument is the same.
\end{rem}

In many instances where we can establish a closed range estimate (e.g., \cite{Hor65}, \cite{Sha85}, \cite{HaRa15}, \cite{ChHa18}), there is also sufficient information to prove that the space of
harmonic $(0,q)$-forms
$\H_{0,q}(\Om)=\{0\}$ (the $q=n-1$ case on the annulus being a notable exception,
as $\dbar$ has closed range but the space of harmonic forms is infinite dimensional \cite{Hor04}),
hence the hypothesis in the next corollary is well-motivated.
\begin{cor}\label{cor: no harmonic forms, solvability in C^infty}
Let $\Om\subset\C^n$ be a bounded smooth domain which admits the family of strong closed range estimates \eqref{eq:temp_strong_closed_range_2}
for some smooth function $\varphi$. Then
\begin{enumerate}[1.]
\item $L^{2,m}_{0,q}(\Omega)\cap\ker\dbar$ is dense in $L^{2,s}_{0,q}(\Om)\cap\ker\dbar$ for any $m>s\geq 0$.
\item If, in addition, $\H_{0,q}(\Om) = \H_{0,q}(\Om,t\varphi) = \{0\}$, then the $\dbar$-problem is solvable
in $C^\infty_{0,\tilde q}(\bar\Om)$ if $\tilde q = q$ or $q-1$. Namely, if $f \in C^\infty_{0,\tilde q+1}(\bar\Om)$ is $\dbar$-closed, then
then there exists $u\in C^\infty_{0,\tilde q}(\bar\Om)$ so that $\dbar u = f$.
\end{enumerate}
\end{cor}

%
%
\section{Notation}\label{sec:notation}

\subsection{\texorpdfstring{$L^2$}{L2} spaces}
Let $\Om\subset\C^n$ be a bounded, $C^m$ domain with $C^m$ defining function $\rho$, $m\geq 2$. Let $\vp$ be a $C^2$ function defined near the closure of $\Om$.
We denote the $L^2$-inner product on $L^2(\Omega,e^{-\phi})$ by
\[
(f,g)_{L^2(\Om,\vp)} = \int_\Omega f \bar g\, e^{-\vp} dV.
\]
We denote
the induced surface area measure on $\bd\Omega$ by $d\sigma$. Also $\| f \|_{L^2(\Om,\vp)}^2 = \int_\Omega |f|^2 e^{-\vp}\, dV$ and if $\vp=0$, we suppress the $\vp$ in the norm.

\subsection{The \texorpdfstring{$\dbar$}{dbar} operator}
Let $\I_q = \{ (i_1,\dots,i_q)\in \N^n : 1 \leq i_1 < \cdots < i_q\leq n \}$. For $I\in\I_{q-1}$, $J\in\I_q$, and $1\leq j \leq n$, let
$\ep^{jI}_{J} = (-1)^{|\sigma|}$ if $\{j\} \cup I = J$ as sets and $|\sigma|$ is the length of the permutation that takes $\{j\}\cup I$ to $J$. Set
$\ep^{jI}_J=0$ otherwise. We use the standard notation that if $u = \sum_{J\in\I_q} u_J\, d\z_J$, then
\[
u_{jI} = \sum_{J\in\I_q} \ep^{jI}_J u_J.
\]

The $\dbar$-operator on $(0,q)$-forms is defined as follows: $\dbar:L^2_{0,q}(\Omega, e^{-t|z|^2}) \to L^2_{0,q+1}(\Omega, e^{-t|z|^2})$ and if
$f = \sum_{J\in\I_q} f_J\, d\z_J$, then
\[
\dbar f = \sum_{\atopp{J\in\I_q}{K\in\I_{q+1}}}\sum_{k=1}^n \ep^{kJ}_K \frac{\p f_J}{\p\z_k}\, d\z_K
\]
We let $\dbars_\vp$ denote the $L^2$-adjoint of $\dbar$ in $L^2_{0,q}(\Om,\vp)$ and denote the weighted $\dbar$-Neumann Laplacian by $\Box_\vp = \dbar\dbars_\vp + \dbars_\vp\dbar$.
If it exists, the inverse to $\Box_\vp$ on $(0,q)$-forms on the orthogonal complement to $\ker\Box_\vp$ is called the \emph{$\dbar$-Neumann operator} and is denoted by $N^\vp_q$.

We use the notation $\H_{0,q}(\Om,\vp)$ for the space
of \emph{$L^2_{0,q}(\Om,\vp)$-harmonic forms}, that is, $\H_{0,q}(\Om,\vp) = \ker(\dbar)\cap\ker(\dbars_\vp)$. We also let $H^q_\vp : L^2_{0,q}(\Om,\vp) \to \H_{0,q}(\Om,\vp)$ denote the orthogonal projection.

\subsection{CR geometry}
The induced CR-structure on $\bd\Omega$  at $z\in\bd\Omega$ is
\[
T^{1,0}_z(\bd\Omega)  = \{ L\in T^{1,0}(\C) : \p\rho(L)=0 \},
\]
where $\rho$ is an arbitrary $C^1$ defining function for $\Omega$.  We denote the exterior algebra generated by these spaces by $T^{p,q}(\bd\Omega)$ and its dual by $\Lambda^{p,q}(\bd\Omega)$.
If we normalize $\rho$ so that $|d\rho|=1$ on $\bd\Omega$, then the \emph{normalized Levi form $\opL$} is the real element of $\Lambda^{1,1}(\bd\Omega)$ defined by
\[
\opL(-i L\wedge \LL) = i\p\dbar\rho(-iL\wedge\LL)
\]
for any $L\in T^{1,0}(\bd\Omega)$.

In the case that $U$ is a small neighborhood of (say) 0, and we write $\Om\cap U$
\[
\Om \cap U = \{(z',x_n+iy_n)\in\C^{n-1}\times\C : y_n > \rho_1(z',x_n)\},
\]
where $\rho_1$ is a $C^2$ function satisfying $\rho_1(0)=0$ and $\nabla\rho_1(0)=0$, then we can identify the normalized Levi form at 0
with the $(n-1)\times(n-1)$ matrix $\big(\frac{\p^2\rho_1}{\p z_j\p \z_k}(0)\big)$ (see \eqref{eq:normalized_hessian} below).

\subsection{\texorpdfstring{$L^2$}{L2} Sobolev spaces}
We define a Sobolev $W^1$ norm that is adapted to the theory for the weighted $\dbar$-Neumann operator.  For $f\in C^\infty_0(\Omega)$ we define
\[
  \norm{f}_{W^1(\Omega,\varphi)}^2=\norm{f}^2_{L^2(\Om,\vp)}+\sum_{j=1}^n\norm{\frac{\partial f}{\partial\bar z_j}}^2_{L^2(\Om,\vp)}
  +\sum_{j=1}^n\norm{e^{\varphi}\frac{\partial}{\partial z_j}\left(e^{-\varphi}f\right)}^2_{L^2(\Om,\vp)}.
\]
As usual, we define $W^1_0(\Omega,\varphi)$ to be the completion of $C^\infty_0(\Omega)$ with respect to this norm.  Note that if we integrate by parts in the $L^2(\Omega,\varphi)$ norm, we obtain the adjoint relation
\[
  \left(\frac{\partial}{\partial\bar z_j}\right)^*_{\varphi}=-e^{\varphi}\frac{\partial}{\partial z_j}\left(e^{-\varphi}\right)=-\frac{\partial}{\partial z_j}+\frac{\partial\varphi}{\partial z_j}.
\]
This motivates the decomposition used in our definition of $W^1(\Omega,\varphi)$.  On bounded domains (or, more generally, domains on which $\varphi$ and $|\nabla\varphi|$ are uniformly bounded), $W^1_0(\Omega)=W^1_0(\Omega,\varphi)$.  On unbounded domains, the theory for such norms has been studied extensively in \cite{HaRa14} and \cite{Has17}, for example.  We now define $W^{-1}(\Omega,\varphi)$ to be the dual of $W^1_0(\Omega,\varphi)$ with respect to $L^2(\Omega,\varphi)$.

We let $L^{2,k}(\Om,\vp)$ denote the usual weighted $L^2$-Sobolev spaces, namely,
\[
\|f\|_{L^{2,k}(\Om,\vp)}^2 = \sum_{|\alpha|\leq k} \|D^\alpha \vp\|_{L^2(\Om,\vp)}^2.
\]
It is the case that $W^1(\Om,\vp) = L^{2,1}(\Om,\vp)$. It is convenient to use $L^{2,k}(\Om,\vp)$ in the elliptic regularization and hence in the proof of Theorem \ref{thm:regularization}, however
we prefer to use $W^1(\Om,\vp)$ in Lemma \ref{lem:SCRE_relationships} because it produces the most refined estimates.

\subsection{Estimates for the \texorpdfstring{$\dbar$}{dbar}-operator}

\begin{defn}
\label{defn:strong_closed_range}
  Let $\Omega\subset\mathbb{C}^n$ be a domain with $C^2$ boundary.  We say that $\Omega$ admits
 a \emph{family of strong closed range estimates} for some $1\leq q\leq n-1$ if there exists a weight function $\varphi\in C^2(\overline\Omega)$ and constants $C_q>0$ and $t_0>0$ such that for every $t\geq t_0$ there exists a cutoff function $\chi_t\in C^\infty_0(\Omega)$ such that $\lim_{t\rightarrow\infty}\frac{\norm{\chi_t}^2_{C^1(\Omega)}}{t^3}=0$ and
        \begin{equation}
        \label{eq:strong_closed_range}
          \norm{\dbar f}^2_{L^2(\Om,t\vp)}+\norm{\dbar^*_{t\varphi} f}^2_{L^2(\Om,t\vp)}+\norm{\chi_t f}^2_{L^2(\Om,t\vp)}\geq t C_q\norm{f}^2_{L^2(\Om,t\vp)}
        \end{equation}
        for all $f\in L^2_{(0,q)}(\Omega)\cap\dom\dbar\cap\dom\dbar^*_{t\varphi}$.  For $p\in \partial\Omega$, we say that $\Omega$ admits a \emph{family of strong closed range estimates near $p$} if, in addition to the above, there exists a family of open neighborhoods $U_t$ of $p$ such that $\lim_{t\rightarrow\infty}t\sup_{z\in U_t}|z-p|^2=\infty$ and \eqref{eq:strong_closed_range} holds for all $f\in L^2_{(0,q)}(\Omega)\cap\dom\dbar\cap\dom\dbar^*_{t\varphi}$ supported in $U_t\cap\overline\Omega$.
\end{defn}
\begin{rem}
  We could also define a family of strong closed range estimates for $(p,q)$-forms with $1\leq p\leq n$, but the presence of $p>0$ does not impact the theory in any way, so we omit this case.
\end{rem}

Closed range, in general, is not a local property.  However, we note that strong closed range estimates localize in the following sense:
\begin{lem}
  Let $\Omega\subset\mathbb{C}^n$ be a bounded domain with $C^2$ boundary.  For some $1\leq q\leq n-1$, $\Omega$ admits a family of strong closed range estimates if and only if $\Omega$ admits a family of strong closed range estimates for every $p\in\partial\Omega$.
\end{lem}

\begin{proof}
To see that global estimates imply local estimates, we simply let $U_t$ be a neighborhood of $\overline\Omega$ that is independent of $t$.  For the converse, let $\psi\in C^\infty(\mathbb{R})$ be a non-decreasing function such that $\psi(x)=0$ for all $x\leq 0$ and $\psi(x)=1$ for all $x\geq 1$.  For $r>0$ and $p\in\mathbb{C}^n$, set $\tilde\xi(z)=\psi\left(\frac{2(r^2-|z-p|^2)}{r^2}\right)$.  Then $\supp\tilde\xi=\overline{B(p,r)}$, $\tilde\xi\equiv 1$ in a neighborhood of $p$, and $|\nabla\tilde\xi|\leq O(r^{-1})$.

Cover $\partial\Omega$ with a finite collection of neighborhoods $U_{j,t}$ satisfying the local definition of strong closed range estimates with cutoff functions $\chi_{j,t}$.  We may assume that $U_{j,t}=B(p_{j,t},r_{j,t})$ where $r_{j,t}^{-2}\leq o(t)$.  If we let $\tilde\xi_{j,t}$ denote the cutoff function defined in the previous paragraph for $B(p_{j,t},r_{j,t})$, then $\xi_{j,t}=\frac{\tilde\xi_{j,t}}{\sum_k\tilde\xi_{k,t}}$ defines a partition of unity in some neighborhood of $\partial\Omega$ satisfying $|\nabla\xi_{j,t}|^2\leq o(t)$.  Hence, if we set $f_{j,t}=f\xi_{j,t}$, then
  \[
    \abs{\dbar f_{j,t}}^2\leq 2\abs{\xi_{j,t}\dbar f}^2+o(t\abs{f}^2).
  \]
  Thus, we may decompose $f=\sum_j f_j$, apply \eqref{eq:strong_closed_range} to each $f_j$, and patch the resulting estimates with error terms that can be absorbed by taking $t$ sufficiently large.  We complete the partition of unity of $\overline\Omega$ using $\xi_{0,t}=1-\sum_j\xi_{j,t}$, and note that we can choose $\chi_t$ to be a constant multiple of $\sqrt{\sum_j\chi_{j,t}^2\xi_{j,t}^2+t C_q\xi_{0,t}^2}$.
\end{proof}

\begin{defn}
\label{defn:compactness_estimate}
  Let $\Omega\subset\mathbb{C}^n$.  We say that $\Omega$ admits a \emph{compactness estimate} for some $1\leq q\leq n$ if for every $\eps>0$ there exists a constant $C_\eps>0$ such that
        \begin{equation}
        \label{eq:compactness}
          \eps\big(\norm{\dbar f}_{L^2(\Om)}^2+\norm{\dbar^* f}_{L^2(\Om)}^2\big)+C_\eps\norm{f}^2_{W^{-1}(\Omega)}\geq \norm{f}_{L^2(\Om)}^2
        \end{equation}
        for all $f\in L^2_{(0,q)}(\Omega)\cap\dom\dbar\cap\dom\dbar^*$.
\end{defn}

We call this a compactness estimates because \eqref{eq:compactness} is equivalent to compactness of the $\bar\partial$-Neumann operator (see Proposition 4.2 in \cite{Str10}).

\section{Sufficient Conditions for Strong Closed Range Estimates}

In many settings, it is more natural to replace the term $\norm{\chi_t f}^2_{L^2(\Om,t\vp)}$ with a large multiple of the Sobolev norm $\norm{f}^2_{W^{-1}(\Omega,t\varphi)}$.  The families of estimates in Lemma \ref{lem:SCRE_relationships} are all candidates for our definition of strong closed range estimates; this lemma shows that the family we have chosen ($(4)$ in Lemma \ref{lem:SCRE_relationships}) is \textit{a priori} the weakest.

\begin{lem}
\label{lem:SCRE_relationships}
  Let $\Omega\subset\mathbb{C}^n$ be a bounded domain with Lipschitz boundary and Lipschitz defining function $\rho$.  Let $1\leq q\leq n-1$ and $\varphi\in C^2(\overline\Omega)$.  For the following families of estimates, we have $(1)\Rightarrow(2)\Rightarrow(3)\Rightarrow(4)$ and $(4)\Rightarrow (3)$.
  \begin{enumerate}
    \item There exist $C_q>0$ and $t_0>0$ such that for every $t\geq t_0$ there exists a cutoff function $\chi_t\in C^\infty_0(\Omega)$ such that $\lim_{t\rightarrow\infty}\frac{\norm{\chi_t}^4_{C^1(\Omega)}}{t^3}=0$ and
        \begin{equation}
        \label{eq:temp_strong_closed_range_1}
          \norm{\dbar f}^2_{L^2(\Om,t\vp)}+\norm{\dbar^*_{t\varphi} f}^2_{L^2(\Om,t\vp)}+\norm{\chi_t f}^2_{L^2(\Om,t\vp)}\geq t C_q\norm{f}^2_{L^2(\Om,t\vp)}
        \end{equation}
        for all $f\in L^2_{(0,q)}(\Omega)\cap\dom\dbar\cap\dom\dbar^*_{t\varphi}$.
    \item There exist $C_q>0$ and $t_0>0$ such that for every $t\geq t_0$ there exists a constant $C_t>0$ satisfying $\lim_{t\rightarrow\infty}\frac{C_t}{t^3}=0$ and
        \begin{equation}
        \label{eq:temp_strong_closed_range_2}
          \norm{\dbar f}^2_{L^2(\Om,t\vp)}+\norm{\dbar^*_{t\varphi} f}^2_{L^2(\Om,t\vp)}+C_t\norm{f}^2_{W^{-1}(\Omega,t\varphi)}\geq t C_q\norm{f}^2_{L^2(\Om,t\vp)}
        \end{equation}
        for all $f\in L^2_{(0,q)}(\Omega)\cap\dom\dbar\cap\dom\dbar^*_{t\varphi}$.
    \item There exist $C_q>0$ and $t_0>0$ such that for every $t\geq t_0$ there exists a constant $C_t>0$ satisfying $\lim_{t\rightarrow\infty}\frac{C_t}{t^3}=0$ and
        \begin{equation}
        \label{eq:temp_strong_closed_range_3}
          \norm{\dbar f}^2_{L^2(\Om,t\vp)}+\norm{\dbar^*_{t\varphi} f}^2_{L^2(\Om,t\vp)}+C_t\norm{(-\rho)f}^2_{L^2(\Om,t\vp)}\geq t C_q\norm{f}^2_{L^2(\Om,t\vp)}
        \end{equation}
        for all $f\in L^2_{(0,q)}(\Omega)\cap\dom\dbar\cap\dom\dbar^*_{t\varphi}$.
    \item There exist $C_q>0$ and $t_0>0$ such that for every $t\geq t_0$ there exists a cutoff function $\chi_t\in C^\infty_0(\Omega)$ such that $\lim_{t\rightarrow\infty}\frac{\norm{\chi_t}^2_{C^1(\Omega)}}{t^3}=0$ and
        \begin{equation}
        \label{eq:temp_strong_closed_range_4}
          \norm{\dbar f}^2_{L^2(\Om,t\vp)}+\norm{\dbar^*_{t\varphi} f}^2_{L^2(\Om,t\vp)}+\norm{\chi_t f}^2_{L^2(\Om,t\vp)}\geq t C_q\norm{f}^2_{L^2(\Om,t\vp)}
        \end{equation}
        for all $f\in L^2_{(0,q)}(\Omega)\cap\dom\dbar\cap\dom\dbar^*_{t\varphi}$.
  \end{enumerate}
\end{lem}

\begin{rem}
A careful analysis of the proof reveals that the condition on $\chi_t$ in $(1)$ can be relaxed to
    \[
      \lim_{t\rightarrow\infty}\frac{\norm{\chi_t}^4_{L^\infty(\Omega)}}{t^3}=0\text{, }\sup_{\{t\geq t_0\}}\frac{\norm{\nabla\chi_t}^2_{L^\infty(\Omega)}}{t\norm{\chi_t}^2_{L^\infty(\Omega)}}<\infty,
    \]
    and the condition on $\chi_t$ in $(4)$ can be relaxed to
    \[
      \lim_{t\rightarrow\infty}\frac{\norm{\chi_t}^2_{L^\infty(\Omega)}}{t^3}=0\text{, }\sup_{\{t\geq t_0\}}\frac{\norm{\nabla\chi_t}^2_{L^\infty(\Omega)}}{\norm{\chi_t}^2_{L^\infty(\Omega)}}<\infty.
    \]
    This requires replacing $\norm{\chi_t}^4_{C^1(\Omega)}$ in \eqref{eq:chi_t_f_estimate} with $\norm{\chi_t}^4_{L^\infty(\Omega)}$.
\end{rem}

\begin{proof}
To see that $(1)$ implies $(2)$, we will need to use the interior regularity for the $\bar\partial$-Neumann problem.  Let $f\in C^1_{(0,q)}(\Omega)\cap\dom\dbar^*_{t\varphi}$.  By definition,
  \[
    \norm{\chi_t f}_{L^2(\Omega,t\varphi)}^2\leq\norm{\chi_t^2 f}_{W^1(\Omega,t\varphi)}\norm{f}_{W^{-1}(\Omega,t\varphi)}.
  \]
  For $\eps>0$ to be chosen later, a small constant/large constant estimate gives us
  \begin{equation}
  \label{eq:chi_t_f_estimate}
    \norm{\chi_t f}_{L^2(\Omega,t\varphi)}^2\leq\frac{\eps}{2\norm{\chi_t}^4_{C^1(\Omega)}}\norm{\chi_t^2 f}_{W^1(\Omega,t\varphi)}^2+\frac{\norm{\chi_t}^4_{C^1(\Omega)}}{2\eps}\norm{f}_{W^{-1}(\Omega,t\varphi)}^2.
  \end{equation}
  To estimate $\norm{\chi_t^2 f}_{W^1(\Omega,t\varphi)}^2$, we observe that integration by parts gives us
  \[
    \norm{\left(\frac{\partial}{\partial\bar z_j}\right)^*_{t\varphi}(\chi_t^2 f)}^2_{L^2(\Omega,t\varphi)}=\int_\Omega\left<\frac{\partial}{\partial\bar z_j}\left(\frac{\partial}{\partial\bar z_j}\right)^*_{t\varphi}(\chi_t^2 f),\chi_t^2 f\right>e^{-t\varphi}dV.
  \]
  Since
  \[
    \left[\frac{\partial}{\partial\bar z_j},\left(\frac{\partial}{\partial\bar z_j}\right)^*_{t\varphi}\right](\chi_t^2 f)=t\frac{\partial^2\varphi}{\partial z_j\partial\bar z_j}\chi_t^2 f,
  \]
  a second integration by parts will give us
  \begin{multline*}
    \norm{\left(\frac{\partial}{\partial\bar z_j}\right)^*_{t\varphi}(\chi_t^2 f)}^2_{L^2(\Omega,t\varphi)}=\\
    \norm{\frac{\partial}{\partial\bar z_j}(\chi_t^2 f)}^2_{L^2(\Omega,t\varphi)}+t\int_\Omega\left<\frac{\partial^2\varphi}{\partial z_j\partial\bar z_j}\chi_t^2 f,\chi_t^2 f\right>e^{-t\varphi}dV.
  \end{multline*}
  Hence,
  \begin{multline*}
    \norm{\chi_t^2 f}_{W^1(\Omega,t\varphi)}^2=\\
    \norm{\chi_t^2 f}_{L^2(\Omega,t\varphi)}^2+2\sum_{j=1}^n\norm{\frac{\partial}{\partial\bar z_j}(\chi_t^2 f)}^2_{L^2(\Omega,t\varphi)}+t\int_\Omega\left<\sum_{j=1}^n\frac{\partial^2\varphi}{\partial z_j\partial\bar z_j}\chi_t^2 f,\chi_t^2 f\right>e^{-t\varphi}dV
  \end{multline*}
  Since $\chi_t^2 f$ is compactly supported, we can use the Morrey-Kohn-H\"ormander identity (see Proposition 4.3.1 in \cite{ChSh01}, for example) with no boundary term to show
  \begin{multline*}
    \norm{\chi_t^2 f}_{W^1(\Omega,t\varphi)}^2\leq
    2\left(\norm{\dbar(\chi_t^2 f)}_{L^2(\Omega,t\varphi)}^2+\norm{\dbar^*_{t\varphi}(\chi_t^2 f)}_{L^2(\Omega,t\varphi)}^2\right)\\
    +O\left((1+t\norm{\varphi}_{C^2(\Omega)})\norm{\chi_t^2 f}_{L^2(\Omega,t\varphi)}^2\right).
  \end{multline*}
Calculating $\dbar(\chi_t^2 f)$ and $\dbars_{t\vp}(\chi_t^2 f)$ and using the inequality $(a+b)^2 \leq 2(a^2+b^2)$ yields the inequality
  \begin{multline*}
    \norm{\chi_t^2 f}_{W^1(\Omega,t\varphi)}^2\leq
    4\left(\norm{\chi_t^2\dbar f}_{L^2(\Omega,t\varphi)}^2+\norm{\chi_t^2\dbar^*_{t\varphi} f}_{L^2(\Omega,t\varphi)}^2\right)\\
    +O\left(\norm{\nabla\chi_t}_{L^\infty(\Omega)}^2\norm{\chi_t f}_{L^2(\Omega,t\varphi)}^2+(1+t\norm{\varphi}_{C^2(\Omega)})\norm{\chi_t^2 f}_{L^2(\Omega,t\varphi)}^2\right),
  \end{multline*}
  or
  \begin{multline*}
    \norm{\chi_t^2 f}_{W^1(\Omega,t\varphi)}^2\leq
    4\norm{\chi_t}^4_{L^\infty(\Omega)}\left(\norm{\dbar f}_{L^2(\Omega,t\varphi)}^2+\norm{\dbar^*_{t\varphi} f}_{L^2(\Omega,t\varphi)}^2\right)\\
    +O\left(\norm{\nabla\chi_t}_{L^\infty(\Omega)}^2\norm{\chi_t}_{L^\infty(\Omega)}^2+(1+t\norm{\varphi}_{C^2(\Omega)})\norm{\chi_t}^4_{L^\infty(\Omega)}\right)\norm{f}_{L^2(\Omega,t\varphi)}^2,
  \end{multline*}
  Substituting this into \eqref{eq:chi_t_f_estimate} and repeatedly using $\norm{\chi_t}_{L^\infty(\Omega)}\leq\norm{\chi_t}_{C^1(\Omega)}$ gives us
  \begin{multline*}
    \norm{\chi_t f}_{L^2(\Omega,t\varphi)}^2\leq 2\eps\left(\norm{\dbar f}_{L^2(\Omega,t\varphi)}^2+\norm{\dbar^*_{t\varphi} f}_{L^2(\Omega,t\varphi)}^2\right)\\
    +O\left(\frac{\eps}{2}\left(1+t\norm{\varphi}_{C^2(\Omega)}\right)\norm{f}_{L^2(\Omega,t\varphi)}^2\right)
    +\frac{\norm{\chi_t}^4_{C^1(\Omega)}}{2\eps}\norm{f}_{W^{-1}(\Omega,t\varphi)}^2.
  \end{multline*}
  We may choose $\eps>0$ sufficiently small so that
  \begin{multline*}
    \norm{\chi_t f}_{L^2(\Omega,t\varphi)}^2\leq \norm{\dbar f}_{L^2(\Omega,t\varphi)}^2+\norm{\dbar^*_{t\varphi} f}_{L^2(\Omega,t\varphi)}^2\\
    +\frac{1}{2}t C_q \norm{f}_{L^2(\Omega,t\varphi)}^2
    +\frac{\norm{\chi_t}^4_{C^1(\Omega)}}{2\eps}\norm{f}_{W^{-1}(\Omega,t\varphi)}^2.
  \end{multline*}
  Substituting this in \eqref{eq:temp_strong_closed_range_1} gives us \eqref{eq:temp_strong_closed_range_2} with $C_t=\frac{\norm{\chi_t}^4_{C^1(\Omega)}}{4\eps}$ and a new constant $\tilde C_q=\frac{1}{4}C_q$.  When $f\in\dom\dbar\cap\dom\dbar^*_{t\varphi}$, we use a standard density result (e.g., Proposition 2.3 in \cite{Str10}).

To see that $(2)$ implies $(3)$, we first recall that there exists a constant $C_\Omega>0$ such that $\norm{(-\rho)^{-1}g}_{L^2(\Omega,t\varphi)}\leq C_\Omega\norm{g}_{W^1(\Om,t\vp)}$ for all $g\in W^1_0(\Om,t\vp)$ (see Theorem 1.4.4.3 in \cite{Gri11}).  Now note that for any $f\in L^2(\Omega,t\varphi)$ such that $(-\rho)f\in L^2(\Omega,t\varphi)$, we have
\begin{equation}
\label{eq:Grisvard}
  \norm{f}_{W^{-1}(\Omega,t\varphi)}=\sup_{g\in W^1_0(\Omega,t\varphi),g\neq 0}\frac{\left<f,g\right>_{L^2(\Omega,t\varphi)}}{\norm{g}_{W^1(\Omega,\varphi)}}\leq C_\Omega\norm{(-\rho)f}_{L^2(\Om,t\vp)}.
\end{equation}

To see that $(3)$ implies $(4)$, we may assume that $\rho$ is a defining function for $\Omega$ that is smooth in the interior of $\Omega$, even if the boundary of $\Omega$ is only $C^2$.  Let $\psi\in C^\infty(\mathbb{R})$ denote a non-decreasing function satisfying $\psi(x)=0$ for all $x\leq 0$ and $\psi(x)=1$ for all $x\geq 1$.  Set
  \[
    \chi_t(z)=\sqrt{C_t}\psi\left(-t\rho(z)-1\right)(-\rho(z)).
  \]
  Then
  \[
    \nabla\chi_t(z)=-\sqrt{C_t}\left(\psi'\left(-t\rho(z)-1\right)t(-\rho(z))+\psi\left(-t\rho(z)-1\right)\right)\nabla\rho(z).
  \]
  Since $\psi'\left(-t\rho(z)-1\right)=0$ whenever $-\rho(z)\geq\frac{2}{t}$, we have $\norm{\chi_t}_{C^1(\Omega)}\leq O(\sqrt{C_t})$, and so
  \[
    \lim_{t\rightarrow\infty}\frac{\norm{\chi_t}^2_{C^1(\Omega)}}{t^3}\leq\lim_{t\rightarrow\infty}O\left(\frac{C_t}{t^3}\right)=0.
  \]
  Since $\psi\left(-t\rho(z)-1\right)\neq 1$ only when $-\rho(z)\leq\frac{2}{t}$, we have
  \begin{align*}
    C_t(-\rho(z))^2&=(\chi_t(z))^2+C_t\left(1-\left(\psi\left(-t\rho(z)-1\right)\right)^2\right)(-\rho(z))^2\\
    &\leq(\chi_t(z))^2+\frac{4 C_t}{t^2},
  \end{align*}
  so \eqref{eq:temp_strong_closed_range_3} implies
  \[
    \norm{\dbar f}^2_{L^2(\Om,t\vp)}+\norm{\dbar^*_{t\varphi} f}^2_{L^2(\Om,t\vp)}+\norm{\chi_t f}^2_{L^2(\Om,t\vp)}\geq\left(t C_q-\frac{4 C_t}{t^2}\right)\norm{f}^2_{L^2(\Om,t\vp)}.
  \]
  For $\tilde t_0>t_0$ sufficiently large, $\frac{1}{2}C_q\geq\frac{4 C_t}{t^3}$ for all $t\geq\tilde t_0$, so $\tilde C_q=\frac{1}{2}C_q$ satisfies $\tilde C_q\leq C_q-\frac{4 C_t}{t^3}$ whenever $t\geq\tilde t_0$, and \eqref{eq:temp_strong_closed_range_4} follows with these new constants $\tilde C_q$ and $\tilde t_0$.

To see that $(4)$ implies $(3)$, we note that since $\chi_t=0$ on $\partial\Omega$, we have
  \[
    \chi_t(z)\leq\norm{\chi_t}_{C^1(\Omega)}\dist(z,\partial\Omega)
  \]
  on $\Omega$.  Since $\dist(z,\partial\Omega)\leq-C_\rho\rho(z)$ on $\Omega$ for some constant $C_\rho>0$, we can let $C_t=\norm{\chi_t}_{C^1(\Omega)}^2C_\rho^2$ and obtain \eqref{eq:temp_strong_closed_range_3} from \eqref{eq:temp_strong_closed_range_4}.

\end{proof}

The motivation for our formulation of strong closed range estimates is the family of estimates that arise naturally in the study of domains with disconnected boundaries (e.g., annuli).  In the estimates constructed in, for example, \cite{Sha85a}, \cite{HaRa15}, or \cite{ChHa18}, different weight functions must be used in a neighborhood of each connected component of the boundary, so a cutoff function must be used to patch these functions together and obtain a global weight function.  This leads to estimates of the form
\[
  \norm{\dbar f}^2_{L^2(\Om,t\vp)}+\norm{\dbar^*_{t\varphi} f}^2_{L^2(\Om,t\vp)}+t C_\chi\norm{\varphi}_{C^2(\Omega)}\norm{\chi f}^2_{L^2(\Om,t\vp)}\geq t C_q\norm{f}^2_{L^2(\Om,t\vp)}
\]
for all $f\in\dom\dbar\cap\dom\dbar^*_{t\varphi}$ for some $\chi\in C^\infty_0(\Omega)$ and $C_\chi>0$.  If we let $\chi_t=\sqrt{t C_\chi\norm{\varphi}_{C^2(\Omega)}}\chi$, then we clearly obtain strong closed range estimates.  However, we also obtain the stronger formulation given by $(1)$ in Lemma \ref{lem:SCRE_relationships}, so in fact all of the families of estimates considered in Lemma \ref{lem:SCRE_relationships} can be obtained in this case.

%
%
\section{Necessary Conditions for Strong Closed Range Estimates}\label{sec:SCRE}

\begin{proof}[Proof of Theorem \ref{thm:main_theorem}]
The beginning of our argument is an adaptation of the argument of  Theorem 3.2.1 in \cite{Hor65}.  By Lemma \ref{lem:SCRE_relationships}, we may assume that we have estimates of the form \eqref{eq:temp_strong_closed_range_3}.

Fix $p\in\partial\Omega$.  After a translation and rotation, we may assume that $p=0$ and there exists some neighborhood $U$ of $p=0$ such that
\[
  \Omega\cap U=\set{z\in U:y_n>\rho_1(z',x_n)},
\]
where $z'=(z_1,\ldots,z_{n-1})$, $z_n=x_n+iy_n$, and $\rho_1$ is a $C^2$ function in some neighborhood of the origin that vanishes to second order at the origin.  Let $\tilde\delta$ denote the signed distance function for $\partial\Omega$.
By \cite[(2.9)]{HaRa13}, since $\abs{\nabla(\rho_1(z',\re z_n)-\im z_n)}=1+O(|z|)$, we have
\begin{equation}
\label{eq:normalized_hessian}
  \frac{\partial^2\tilde\delta}{\partial z_j\partial\bar z_k}(0)=\frac{\partial^2\rho_1}{\partial z_j\partial\bar z_k}(0)
\end{equation}
for all $1\leq j,k\leq n-1$.

Fix $s>0$ and define
\begin{equation}
\label{eq:L_defn_1}
  L_s(z'):=\sum_{j,k=1}^{n-1}\frac{\partial^2\rho_1}{\partial z_j\partial\bar z_k}(0)z_j\bar z_k+s\sum_{j,k=1}^{n-1}\frac{\partial^2\varphi}{\partial z_j\partial\bar z_k}(0)z_j\bar z_k.
\end{equation}
After a unitary change of coordinates, we may assume that
\begin{equation}
\label{eq:L_defn_2}
  L_s(z')=\sum_{j=1}^{n-1}\lambda^s_j|z_j|^2
\end{equation}
for some increasing sequence of real numbers $\{\lambda^s_j\}_{1\leq j\leq n-1}$.

Let $\psi_1\in C^\infty_0(\mathbb{C}^{n-1})$ and $\psi_3\in C^\infty_0(\mathbb{R})$ satisfy $\psi_3\equiv 1$ in a neighborhood of $0$ and $\int_{\mathbb{R}}|\psi_3|^2=1$.  For $x+iy\in\C$, if
we define
\[
  \psi_2(x+iy)=\psi_3(y)\left(\psi_3(x)+iy\psi_3'(x)-\frac{y^2}{2}\psi_3''(x)\right),
\]
then $\psi_2(z)$ is a smooth, compactly supported function on $\mathbb{C}$ satisfying
\begin{equation}
\label{eq:psi_2_derivative}
  \frac{\partial}{\partial\bar z}\psi_2(z)\Big|_{z=x}=0,
\end{equation}
\begin{equation}
\label{eq:psi_2_second_derivative}
  \frac{\partial^2}{\partial\bar z^2}\psi_2(z)\Big|_{z=x}=0,
\end{equation}
and
\begin{equation}
\label{eq:psi_2_integral}
  \int_{\mathbb{R}}\big|\psi_2(x)\big|^2dx=1.
\end{equation}

Let $A(z)$ and $B(z)$ be the holomorphic polynomials
\[
  A(z)=\sum_{j,k=1}^{n-1}\frac{\partial^2\rho_1}{\partial z_j\partial z_k}(0)z_j z_k+\sum_{j=1}^{n-1}2\frac{\partial^2\rho_1}{\partial z_j\partial x_n}(0)z_j z_n+\frac{1}{2}\frac{\partial^2\rho_1}{\partial x_n^2}(0)z_n^2
\]
and
\begin{multline*}
  B(z)=\frac{1}{2}\varphi(0)+\sum_{j=1}^n\frac{\partial\varphi}{\partial z_j}(0)z_j+\sum_{j,k=1}^{n-1}\frac{1}{2}\frac{\partial^2\varphi}{\partial z_j\partial z_k}(0)z_j z_k\\
  +\sum_{j=1}^{n-1}\frac{\partial^2\varphi}{\partial z_j\partial x_n}(0)z_j z_n+\frac{1}{4}\frac{\partial^2\varphi}{\partial x_n^2}(0)z_n^2.
\end{multline*}
Then we have
\begin{equation}
\label{eq:A_Taylor_Series}
  \abs{\rho_1(z)-\re A(z)-
  \sum_{j,k=1}^{n-1}\frac{\partial^2\rho_1}{\partial z_j\partial\bar z_k}(0)z_j\bar z_k}\leq O(|z|^3+|y_n||z|+|y_n|^2),
\end{equation}
and
\begin{equation}
\label{eq:B_Taylor_Series}
  \abs{\varphi(z)-2\re B(z)-\sum_{j,k=1}^{n-1}\frac{\partial^2\varphi}{\partial z_j\partial\bar z_k}(0)z_j\bar z_k}\leq O(|z|^3+|y_n||z|+|y_n|^2).
\end{equation}

For any $\tau>0$ we define a form in $C^\infty_{(0,q)}(\overline\Omega)\cap\dom\dbar^*$ by
\[
  f^\tau(z)=\psi_1(\tau z')\psi_2(\tau z_n)e^{\tau^2(A(z)+2s B(z)+i z_n)}\bigwedge_{j=1}^q\left(d\bar z_j-\left(\frac{\partial\tilde\delta}{\partial z_n}\right)^{-1}\frac{\partial\tilde\delta}{\partial z_j}d\bar z_n\right).
\]
As in H\"ormander's construction, we note that
\begin{equation}
\label{eq:q-form}
  \bigwedge_{j=1}^q\left(d\bar z_j-\left(\frac{\partial\tilde\delta}{\partial z_n}\right)^{-1}\frac{\partial\tilde\delta}{\partial z_j}d\bar z_n\right)\Bigg|_{z=0}=\bigwedge_{j=1}^q d\bar z_j,
\end{equation}
so the term involving $\frac{\partial\tilde\delta}{\partial z_j}$ will vanish in our asymptotic computations.  We introduce the change of coordinates $z_j(\tau)=\tau^{-1}w_j$ for $1\leq j\leq n-1$ and $z_n(\tau)=\tau^{-1}\re w_n+i\tau^{-2}\im w_n$.  Using \eqref{eq:A_Taylor_Series}, we have
\begin{multline}
\label{eq:defining_function_limit}
  \lim_{\tau\rightarrow\infty}\tau^2(\rho_1(z'(\tau),x_n(\tau))-y_n(\tau))=\\
  \re A(w',\re w_n)+\sum_{j,k=1}^{n-1}\frac{\partial^2\rho_1}{\partial z_j\partial\bar z_k}(0)w_j\bar w_k-\im w_n,
\end{multline}
so as $\tau\rightarrow\infty$ in our special coordinates, we will be working on the domain
\begin{equation}
\label{eq:Omega_w}
  \Omega_w=\set{w\in\mathbb{C}^n:\im w_n>\sum_{j,k=1}^{n-1}\frac{\partial^2\rho_1}{\partial z_j\partial\bar z_k}(0)w_j\bar w_k\\+\re A(w',\re w_n)}.
\end{equation}
Furthermore, we may use \eqref{eq:B_Taylor_Series} to check
\begin{equation}
\label{eq:B_limit}
  \lim_{\tau\rightarrow\infty}\tau^2(\varphi(z(\tau))-2\re B(z(\tau)))=\sum_{j,k=1}^{n-1}\frac{\partial^2\varphi}{\partial z_j\partial\bar z_k}(0)w_j\bar w_k.
\end{equation}
Motivated by this, we set $t(\tau)=2s\tau^2$.  For such a value of $t$, we may use \eqref{eq:q-form} and \eqref{eq:B_limit} to show
\begin{multline}
\label{eq:pointwise_limit}
  \lim_{\tau\rightarrow\infty}\abs{f^\tau(z(\tau))}^2 e^{-t(\tau)\varphi(z(\tau))}=\abs{\psi_1(w')}^2\abs{\psi_2(\re w_n)}^2\\
  \exp\left(2\re A(w',\re w_n)-2\im w_n-2s\sum_{j,k=1}^{n-1}\frac{\partial^2\varphi}{\partial z_j\partial\bar z_k}(0)w_j\bar w_k\right),
\end{multline}
so
\begin{multline*}
  \lim_{\tau\rightarrow\infty}\tau^{2n+1}\norm{f^\tau}^2_{L^2(\Om,t(\tau)\vp)}=\int_{\Omega_w}\abs{\psi_1(w')}^2\abs{\psi_2(\re w_n)}^2\\
   \exp\left(2\re A(w',\re w_n)-2\im w_n-2s\sum_{j,k=1}^{n-1}\frac{\partial^2\varphi}{\partial z_j\partial\bar z_k}(0)w_j\bar w_k\right)dV_w.
\end{multline*}
Since $\int_a^\infty e^{-2x}dx=\frac{1}{2}e^{-2a}$ for any $a\in\mathbb{R}$, we may use \eqref{eq:Omega_w} to evaluate this integral with respect to $\im w_n$ and then use \eqref{eq:psi_2_integral} to evaluate this integral with respect to $\re w_n$ to obtain
\begin{equation}
\label{eq:L2_limit}
  \lim_{\tau\rightarrow\infty}\tau^{2n+1}\norm{f^\tau}^2_{L^2(\Om,t(\tau)\vp)}=
  \int_{\mathbb{C}^{n-1}}\frac{1}{2}\abs{\psi_1(w')}^2
   e^{-2L_s(w')}dw'.
\end{equation}
By the same reasoning, we may use \eqref{eq:q-form} to obtain
\begin{multline}
\label{eq:weight_limit}
  \lim_{\tau\rightarrow\infty}\tau^{2n+1}\int_\Omega\sum_{j,k=1}^n
  \sum_{K\in\mathcal{I}_{q-1}}\frac{\partial^2\varphi}{\partial z_j\partial\bar z_k}f^\tau_{jK}\overline{f^\tau_{kK}}e^{-t(\tau)\varphi}dV=\\
  \int_{\mathbb{C}^{n-1}}\frac{1}{2}\abs{\psi_1(w')}^2\sum_{j=1}^q\frac{\partial^2\varphi}{\partial z_j\partial\bar z_j}(0)
   e^{-2 L_s(w')}dw'.
\end{multline}
If instead we integrate \eqref{eq:pointwise_limit} over the boundary, we may use \eqref{eq:Omega_w} and \eqref{eq:psi_2_integral} directly to obtain
\[
  \lim_{\tau\rightarrow\infty}\tau^{2n-1}\int_{\partial\Omega}\abs{f^\tau}^2 e^{-t(\tau)\varphi}d\sigma=\int_{\mathbb{C}^{n-1}}\abs{\psi_1(w')}^2
   e^{-2 L_s(w')}dw'.
\]
Similarly, using \eqref{eq:normalized_hessian} and \eqref{eq:q-form}, we also have
\begin{multline}
\label{eq:boundary_limit}
  \lim_{\tau\rightarrow\infty}\tau^{2n-1}\int_{\partial\Omega}\sum_{j,k=1}^n
  \sum_{K\in\mathcal{I}_{q-1}}\frac{\partial^2\tilde\delta}{\partial z_j\partial\bar z_k}f^\tau_{jK}\overline{f^\tau_{kK}}e^{-t(\tau)\varphi}d\sigma=\\
  \int_{\mathbb{C}^{n-1}}\abs{\psi_1(w')}^2\sum_{j=1}^q\frac{\partial^2\rho_1}{\partial z_j\partial\bar z_j}(0)
   e^{-2 L_s(w')}dw'.
\end{multline}
Using \eqref{eq:L_defn_1} and \eqref{eq:L_defn_2} to add \eqref{eq:weight_limit} and \eqref{eq:boundary_limit}, we obtain
\begin{multline}
\label{eq:combined_limit}
  \lim_{\tau\rightarrow\infty}\tau^{2n-1}\Bigg(\int_{\partial\Omega}\sum_{j,k=1}^n
  \sum_{K\in\mathcal{I}_{q-1}}\frac{\partial^2\tilde\delta}{\partial z_j\partial\bar z_k}f^\tau_{jK}\overline{f^\tau_{kK}}e^{-t(\tau)\varphi}d\sigma\\
  +t\int_\Omega\sum_{j,k=1}^n
  \sum_{K\in\mathcal{I}_{q-1}}\frac{\partial^2\varphi}{\partial z_j\partial\bar z_k}f^\tau_{jK}\overline{f^\tau_{kK}}e^{-t(\tau)\varphi}dV\Bigg)=\\
  \int_{\mathbb{C}^{n-1}}\abs{\psi_1(w')}^2\sum_{j=1}^q\lambda^s_j
   e^{-2 L_s(w')}dw'.
\end{multline}

For $1\leq k\leq n-1$, we compute
\begin{multline*}
  \frac{\partial f^\tau}{\partial\bar z_k}(z)=
  \tau\frac{\partial \psi_1}{\partial\bar z_k}(\tau z')\psi_2(\tau z_n)e^{\tau^2(A(z)+2s B(z)+i z_n)}\bigwedge_{j=1}^q\left(d\bar z_j-\left(\frac{\partial\tilde\delta}{\partial z_n}\right)^{-1}\frac{\partial\tilde\delta}{\partial z_j}d\bar z_n\right)\\
  +\psi_1(\tau z')\psi_2(\tau z_n)e^{\tau^2(A(z)+2s B(z)+i z_n)}\frac{\partial}{\partial\bar z_k}\bigwedge_{j=1}^q\left(d\bar z_j-\left(\frac{\partial\tilde\delta}{\partial z_n}\right)^{-1}\frac{\partial\tilde\delta}{\partial z_j}d\bar z_n\right).
\end{multline*}
Furthermore,
\begin{multline}
\label{eq:f_tau_derivative_n}
  \frac{\partial f^\tau}{\partial\bar z_n}(z)=
  \tau\psi_1(\tau z')\frac{\partial\psi_2}{\partial\bar z_n}(\tau z_n)e^{\tau^2(A(z)+2s B(z)+i z_n)}\bigwedge_{j=1}^q\left(d\bar z_j-\left(\frac{\partial\tilde\delta}{\partial z_n}\right)^{-1}\frac{\partial\tilde\delta}{\partial z_j}d\bar z_n\right)\\
  +\psi_1(\tau z')\psi_2(\tau z_n)e^{\tau^2(A(z)+2s B(z)+i z_n)}\frac{\partial}{\partial\bar z_n}\bigwedge_{j=1}^q\left(d\bar z_j-\left(\frac{\partial\tilde\delta}{\partial z_n}\right)^{-1}\frac{\partial\tilde\delta}{\partial z_j}d\bar z_n\right).
\end{multline}
Hence, using \eqref{eq:psi_2_derivative} and observing that the second term in each derivative is uniformly bounded in $\tau$, we have
\begin{multline}
\label{eq:pointwise_gradient_limit}
  \lim_{\tau\rightarrow\infty}\tau^{-2}\sum_{j=1}^n\abs{\frac{\partial}{\partial\bar z_j}f^\tau(z)\Big|_{z=z(\tau)}}^2 e^{-t(\tau)\varphi(z(\tau))}=\sum_{j=1}^{n-1}\abs{\frac{\partial}{\partial\bar w_j}\psi_1(w')}^2\abs{\psi_2(\re w_n)}^2\\
  \exp\left(2\re A(w',\re w_n)-2\im w_n-2s\sum_{j,k=1}^{n-1}\frac{\partial^2\varphi}{\partial z_j\partial\bar z_k}(0)w_j\bar w_k\right).
\end{multline}
Integrating \eqref{eq:pointwise_gradient_limit} as before, we obtain
\begin{equation}
\label{eq:gradient_limit}
  \lim_{\tau\rightarrow\infty}\tau^{2n-1}\sum_{j=1}^n\norm{\frac{\partial}{\partial\bar z_j}f^\tau}^2_{L^2(\Om,t(\tau)\vp)}=
  \int_{\mathbb{C}^{n-1}}\frac{1}{2}\sum_{j=1}^{n-1}\abs{\frac{\partial}{\partial\bar w_j}\psi_1(w')}^2
   e^{-2 L_s(w')}dw'.
\end{equation}

Recall the Morrey-Kohn-H\"ormander identity:
\begin{multline}
\label{eq:MKH_identity}
  \norm{\dbar f^\tau}^2_{L^2(\Om,t\vp)}+\norm{\dbar^*_{t\varphi} f^\tau}^2_{L^2(\Om,t\vp)}=\sum_{j=1}^n\norm{\frac{\partial}{\partial\bar z_j}f^\tau}^2_{L^2(\Om,t\vp)}\\
  +\int_{\partial\Omega}\sum_{j,k=1}^n\sum_{K\in\mathcal{I}_{q-1}}\frac{\partial^2\tilde\delta}{\partial z_j\partial\bar z_k}f^\tau_{jK}\overline{f^\tau_{kK}}e^{-t\varphi}d\sigma\\
  +t\int_\Omega\sum_{j,k=1}^n\sum_{K\in\mathcal{I}_{q-1}}\frac{\partial^2\varphi}{\partial z_j\partial\bar z_k}f^\tau_{jK}\overline{f^\tau_{kK}}e^{-t\varphi}dV
\end{multline}
Note that if $\rho$ is an arbitrary defining function for $\Omega$, then $\rho(z)=h(z)(\rho_1(z)-\im z_n)$ for a bounded function $h$ that is uniformly bounded away from zero, so \eqref{eq:defining_function_limit}, the fact that $t(\tau) = 2s\tau^2$, and the hypothesis that $\frac{C_t}{t^3}\rightarrow 0$ imply that
\[
  \lim_{\tau\rightarrow\infty}\frac{C_{t(\tau)}(-\rho(z(\tau)))^2}{t(\tau)}=\lim_{\tau\rightarrow\infty}\frac{C_{t(\tau)}(-2s\tau^2\rho(z(\tau)))^2}{(t(\tau))^3}=0.
\]
As a result,
\begin{equation}
\label{eq:compact_term_limit}
  \lim_{\tau\rightarrow\infty}\tau^{2n-1}C_{t(\tau)}\norm{(-\rho)f^\tau}^2_{L^2(\Om,t(\tau)\vp)}=\lim_{\tau\rightarrow\infty}\tau^{2n+1}\frac{2sC_{t(\tau)}\norm{(-\rho)f^\tau}^2_{L^2(\Om,t(\tau)\vp)}}{t(\tau)}=0.
\end{equation}
Hence, combining \eqref{eq:temp_strong_closed_range_3} with \eqref{eq:MKH_identity} gives us
\begin{multline*}
  t(\tau) C_q\norm{f^\tau}^2_{L^2(\Om,t(\tau)\vp)}\\\leq\sum_{j=1}^n\norm{\frac{\partial}{\partial\bar z_j}f^\tau}^2_{L^2(\Om,t(\tau)\vp)}
  +\int_{\partial\Omega}\sum_{j,k=1}^n\sum_{K\in\mathcal{I}_{q-1}}\frac{\partial^2\tilde\delta}{\partial z_j\partial\bar z_k}f^\tau_{jK}\overline{f^\tau_{kK}}e^{-t(\tau)\varphi}d\sigma\\
  +t(\tau)\int_\Omega\sum_{j,k=1}^n\sum_{K\in\mathcal{I}_{q-1}}\frac{\partial^2\varphi}{\partial z_j\partial\bar z_k}f^\tau_{jK}\overline{f^\tau_{kK}}e^{-t(\tau)\varphi}dV +C_{t(\tau)}\norm{(-\rho)f^\tau}^2_{L^2(\Om,t(\tau)\vp)}.
\end{multline*}
Multiplying this by $\tau^{2n-1}$ and taking a limit using \eqref{eq:L2_limit}, \eqref{eq:combined_limit}, \eqref{eq:gradient_limit}, and \eqref{eq:compact_term_limit}, we obtain
\begin{multline*}
  s C_q\int_{\mathbb{C}^{n-1}}\abs{\psi_1(w')}^2
   e^{-2 L_s(w')}dw'\\
   \leq\int_{\mathbb{C}^{n-1}}\left(\frac{1}{2}\sum_{j=1}^{n-1}\abs{\frac{\partial}{\partial\bar w_j}\psi_1(w')}^2
   +\abs{\psi_1(w')}^2\sum_{j=1}^q\lambda^s_j\right)
   e^{-2 L_s(w')}dw'.
\end{multline*}
Rearranging terms, we obtain
\begin{multline*}
  \left(s C_q-\sum_{j=1}^q\lambda^s_j\right)\int_{\mathbb{C}^{n-1}}\abs{\psi_1(w')}^2
   e^{-2 L_s(w')}dw'\\
   \leq\frac{1}{2}\int_{\mathbb{C}^{n-1}}\sum_{j=1}^{n-1}\abs{\frac{\partial}{\partial\bar w_j}\psi_1(w')}^2
   e^{-2 L_s(w')}dw'.
\end{multline*}
From Lemma 3.2.2 in \cite{Hor65}, we immediately obtain
\begin{equation}
\label{eq:eigenvalue_estimate}
  s C_q-\sum_{j=1}^q\lambda^s_j\leq\sum_{j=1}^{n-1}\max(-\lambda^s_j,0).
\end{equation}

Now, let $n^s_-$ denote the number of $j$ for which $\lambda^s_j<0$ and let $n^s_+$ denote the number of $j$ for which $\lambda^s_j>0$.  Since we have assumed that the $\{\lambda^s_j\}$ are arranged in increasing order, \eqref{eq:eigenvalue_estimate} implies
\begin{equation}
\label{eq:eigenvalues_estimate_refined}
  s C_q\leq\sum_{j=1}^q\lambda^s_j-\sum_{j=1}^{n^s_-}\lambda^s_j.
\end{equation}
If $n^s_-\leq q$ and $n^s_+\leq n-q-1$, then $s C_q\leq 0$, so we obtain an immediate contradiction.  Hence, either $n^s_-\geq q+1$ or $n^s_+\geq n-q$.

We now think of $s$ as a free parameter rather than a fixed constant.  Note that $n^s_-(z)$ and $n^s_+(z)$ are lower semicontinuous functions for $(s,z)\in\mathbb{R}^+\times \partial\Omega$.  Hence, if $n^s_-(z)>q$ at a point in $\mathbb{R}^+\times \partial\Omega$, then this condition also holds for a neighborhood of that point, with the analogous statement for $n^s_+>n-q-1$.  We conclude that for every connected component $S\subset \partial\Omega$, we must have $n^s_-(z)\geq q+1$ for all $z\in S$ and $s>0$ or $n^s_+\geq n-q$ for all $z\in S$ and $s>0$.

We first consider the case in which $n^s_+(z)\geq n-q$ for all $z\in S\subset \partial\Omega$ and $s>0$.  Then $n^s_-(z)<q$, so \eqref{eq:eigenvalues_estimate_refined} gives us
\[
  s C_q\leq\sum_{j=n^s_-(z)+1}^q\lambda^s_j(z)\leq(q-n^s_-(z))\lambda_q^s(z)\leq q\lambda_q^s(z),
\]
for all $z\in S$ and $s>0$, where the final inequality relies on the fact that the previous inequality guarantees $\lambda_q^s(z)>0$.  Letting $s\rightarrow 0$, we see that $\lambda_q^0(z)\geq 0$, i.e., the Levi form has at least $n-q$ nonnegative eigenvalues at $z$.  Since
\[
  \lim_{s\rightarrow\infty}\frac{\lambda_q^s(z)}{s}\geq\frac{C_q}{q},
\]
we see that the restriction of $i\ddbar\varphi$ to $T^{1,0}(\partial\Omega)\times T^{0,1}(\partial\Omega)$ must have at least $n-q$ eigenvalues bounded below by $\frac{C_q}{q}$.

We next suppose $n^s_-(z)\geq q+1$ for all $z\in S\subset \partial\Omega$ and $s>0$.  Then \eqref{eq:eigenvalues_estimate_refined} gives us
\[
  s C_q\leq-\sum_{j=q+1}^{n^s_-}\lambda^s_j\leq-(n^s_-(z)-q)\lambda_{q+1}^s(z)\leq -(n-1-q)\lambda_{q+1}^s(z),
\]
for all $z\in S$ and $s>0$, where the final inequality relies on the fact that the previous inequality guarantees $\lambda_{q+1}^s(z)<0$.  Letting $s\rightarrow 0$, we see that $\lambda_{q+1}^0(z)\leq 0$, i.e., the Levi form has at least $q+1$ nonpositive eigenvalues at $z$.  Since
\[
  \lim_{s\rightarrow\infty}\frac{\lambda_{q+1}^s(z)}{s}\leq-\frac{C_q}{n-1-q},
\]
we see that the restriction of $i\ddbar\varphi$ to $T^{1,0}(\partial\Omega)\times T^{0,1}(\partial\Omega)$ must have at least $q+1$ eigenvalues bounded above by $\frac{C_q}{n-q-1}$.  We have now proved the global version of Theorem \ref{thm:main_theorem}.

If we only have local estimates, it suffices to note that the support of $f^\tau$ is contained in a neighborhood of radius $O(\tau^{-1})$, and so $|z-p|^2\leq O(t^{-1})$ when $f^\tau(z)\neq 0$.  Hence, for $t$ sufficiently large, $f^\tau$ is supported in $U_t$, and the rest of the proof follows to obtain pointwise information at $p\in\partial\Omega$.
\end{proof}

\begin{proof}[Proof of Corollary \ref{cor:pseudoconvexity}]
  It suffices to note that since $\Omega$ is bounded, if we write $\Omega=\Omega_0\backslash\bigcup_{j\in J}\overline{\Omega_j}$, then $\partial\Omega_0$ and $\partial\Omega_j$ must each admit at least one strictly convex point for all $j\in J$.  Since $\Omega_0$ admits a strictly convex point, $\partial\Omega_0$ must satisfy $(1)$ in Theorem \ref{thm:main_theorem}.  For $j\in J$, $\partial\Omega_j$ admits at least one strictly convex point, so $\partial\Omega_j$ viewed as a component of $\partial\Omega$ admits at least one strictly concave point, and hence satisfies $(2)$ in Theorem \ref{thm:main_theorem}.
\end{proof}

%
%
\section{Applications}\label{sec:apps}

\subsection{Subelliptic and Compactness Estimates}

\begin{proof}[Proof of Proposition \ref{prop:subelliptic}]
Suppose that for some $\eta>0$, $\Omega$ admits a subelliptic estimate of the form
\[
  \norm{u}^2_{W^{\eta}(\Omega)}\leq C\left(\norm{\dbar u}^2_{L^2(\Omega)}+\norm{\dbar^* u}^2_{L^2(\Omega)}\right)
\]
for all $u\in L^2_{(0,q)}(\Omega)\cap\dom\dbar\cap\dom\dbar^*$.  By definition,
\[
  \norm{u}^2_{L^2(\Omega)}\leq\norm{u}_{W^{\eta}(\Omega)}\norm{u}_{W^{-\eta}(\Omega)},
\]
so for any $t>0$ we may use a small constant/large constant inequality to obtain
\[
  \norm{u}^2_{L^2(\Omega)}\leq \frac{1}{2t}\norm{u}_{W^{\eta}(\Omega)}^2+\frac{t}{2}\norm{u}_{W^{-\eta}(\Omega)}^2.
\]
If $\rho$ is a defining function for $\Omega$, then there exists a constant $C_{\Omega,\eta}>0$ such that
\[
  \norm{u}_{W^{-\eta}(\Omega)}\leq C_{\Omega,\eta}\norm{(-\rho)^\eta u}_{L^2(\Omega)},
\]
This follows from Theorem 1.4.4.3 in \cite{Gri11} by a duality argument, as in \eqref{eq:Grisvard}.  Hence,
\[
  \norm{u}^2_{L^2(\Omega)}\leq \frac{1}{2t}\norm{u}_{W^{\eta}(\Omega)}^2+\frac{t C_{\Omega,\eta}^2}{2}\norm{(-\rho)^\eta u}_{L^2(\Omega)}^2.
\]
Substituting our subelliptic estimate yields
\begin{equation}
\label{eq:subelliptic_SCRE}
  \norm{u}^2_{L^2(\Omega)}\leq \frac{C}{2t}\left(\norm{\dbar u}^2_{L^2(\Omega)}+\norm{\dbar^* u}^2_{L^2(\Omega)}\right)+\frac{t C_{\Omega,\eta}^2}{2}\norm{(-\rho)^\eta u}_{L^2(\Omega)}^2.
\end{equation}
We may assume that $\rho$ is a defining function for $\Omega$ that is smooth in the interior of $\Omega$, even if the boundary of $\Omega$ is only $C^2$.  Let $\psi\in C^\infty(\mathbb{R})$ denote a non-decreasing function satisfying $\psi(x)=0$ for all $x\leq 0$ and $\psi(x)=1$ for all $x\geq 1$.  Fix $b>a>0$ and set
  \[
    \chi_t(z)=\frac{t C_{\Omega,\eta}}{\sqrt{C}}\psi\left(\frac{-t^{1/(2\eta)}\rho(z)-a}{b-a}\right)(-\rho(z))^\eta.
  \]
  Since $\psi\left(\frac{-t^{1/(2\eta)}\rho(z)-a}{b-a}\right)\neq 1$ only when $-\rho(z)\leq\frac{b}{t^{1/(2\eta)}}$, we have
  \begin{align*}
    \frac{t^2 C_{\Omega,\eta}^2}{C}(-\rho(z))^{2\eta}&=(\chi_t(z))^2+\frac{t^2 C_{\Omega,\eta}^2}{C}\left(1-\left(\psi\left(\frac{-t^{1/(2\eta)}\rho(z)-a}{b-a}\right)\right)^2\right)(-\rho(z))^{2\eta}\\
    &\leq(\chi_t(z))^2+\frac{t C_{\Omega,\eta}^2b^{2\eta}}{C}.
  \end{align*}
  Substituting in \eqref{eq:subelliptic_SCRE} and rearranging terms, we obtain
  \[
    \frac{t}{C}\left(2-C_{\Omega,\eta}^2b^{2\eta}\right)\norm{u}^2_{L^2(\Omega)}\leq \norm{\dbar u}^2_{L^2(\Omega)}+\norm{\dbar^* u}^2_{L^2(\Omega)}+\norm{\chi_t u}_{L^2(\Omega)}^2.
  \]
  If we choose $b$ sufficiently small, then we may set $C_q=\frac{1}{C}\left(2-C_{\Omega,\eta}^2b^{2\eta}\right)$ and obtain $C_q>0$.  Thus we have an estimate of the form \eqref{eq:strong_closed_range} with $\varphi\equiv 0$, but without the growth condition on $\norm{\chi_t}^2_{C^1(\Omega)}$.  Since $\Omega$ is bounded, we have $\norm{\chi_t}_{L^\infty(\Omega)}\leq O(t)$.  We compute
  \begin{multline*}
    \nabla\chi_t(z)=-\frac{\eta t C_{\Omega,\eta}}{\sqrt{C}}\psi\left(\frac{-t^{1/(2\eta)}\rho(z)-a}{b-a}\right)(-\rho(z))^{\eta-1}\nabla\rho(z)\\
    -\frac{t^{1+1/(2\eta)} C_{\Omega,\eta}}{(b-a)\sqrt{C}}\psi'\left(\frac{-t^{1/(2\eta)}\rho(z)-a}{b-a}\right)(-\rho(z))^\eta\nabla\rho(z).
  \end{multline*}
  Since $\psi\left(\frac{-t^{1/(2\eta)}\rho(z)-a}{b-a}\right)\neq 0$ only when $-\rho(z)\geq\frac{a}{t^{1/(2\eta)}}$, the first term is bounded by $O(t\cdot t^{(1-\eta)/(2\eta)})=O(t^{(1+\eta)/(2\eta)})$.  Since $\psi'\left(\frac{-t^{1/(2\eta)}\rho(z)-a}{b-a}\right)\neq 0$ only when $\frac{b}{t^{1/(2\eta)}}\geq-\rho(z)$, the second term is bounded by $O(t^{1+1/(2\eta)}\cdot t^{-1/2})=O(t^{(1+\eta)/(2\eta)})$.  Consequently,
  \[
    \limsup_{t\rightarrow\infty}\frac{\norm{\chi_t}^2_{C^1(\Omega)}}{t^{(1+\eta)/\eta}}<\infty.
  \]
  Since $\varphi\equiv 0$, the conclusions of Theorem \ref{thm:main_theorem} does not follow, and hence
  \[
    \limsup_{t\rightarrow\infty}\frac{\norm{\chi_t}^2_{C^1(\Omega)}}{t^3}>0.
  \]
\end{proof}

\begin{proof}[Proof of Theorem \ref{thm:compactness}]
Suppose that \eqref{eq:quantitative_compactness} fails.  Then
  \[
    \lim_{\eps\rightarrow 0^+}{\eps}^2 C_{\eps}=0.
  \]
  For any $t>0$, let $\eps=\frac{1}{t}$.  Then we may set $C_t=t C_{\eps}$, $C_q=1$, and $\varphi\equiv 0$ to show that \eqref{eq:compactness} implies \eqref{eq:temp_strong_closed_range_2} (observe that $\frac{C_t}{t^3}=\eps^2 C_\eps$).
By Lemma \ref{lem:SCRE_relationships}, this also implies \eqref{eq:strong_closed_range}.  Hence, Theorem \ref{thm:main_theorem} implies that $i\ddbar\varphi$ has nontrivial eigenvalues, contradicting the fact that $\varphi$ is constant.
\end{proof}

\subsection{Sobolev Estimates}
For the remainder of this note, we concentrate on the implications of \eqref{eq:temp_strong_closed_range_2}.  Note that the following arguments do not require $C_t$ to depend on $t$ in
a prescribed way.

The following lemma appears in \cite{KhRa20}, though it is well-known.
\begin{lem}\label{lem:harm forms, L^2 ests}
Let $1 \leq q \leq n-1$. Suppose that $\Om \subset\C^n$ is a bounded domain. Then the following are equivalent:
\begin{enumerate}
\item The space of harmonic forms $\H_{0,q}(\Om,\phi)$ is finite dimensional and the $L^2$-basic estimate
\begin{equation}\label{eqn:BE for non harm}
\|u\|_{L^2(\Om,\phi)}^2 \leq C\big(\|\dbar u\|_{L^2(\Om,\phi)}^2 + \|\dbars_\phi u \|_{L^2(\Om,\phi)}^2\big)
\end{equation}
holds for all $u \in \Dom(\dbar)\cap\Dom(\dbars_\phi) \cap (\H_{0,q}(\Om,\phi))^\perp$.
\item The $L^2$-basic estimate
\begin{equation}\label{eqn:BE for general}
\|u\|_{L^2(\Om,\phi)}^2 \leq C\big(\|\dbar u\|_{L^2(\Om,\phi)}^2 + \|\dbars_\phi u \|_{L^2(\Om,\phi)}^2\big) + C_\phi\|u\|_{L^{2,-1}(\Om,\phi)}^2
\end{equation}
holds for all $u \in \Dom(\dbar)\cap\Dom(\dbars_\phi)$.
\end{enumerate}
\end{lem}

We now use an elliptic regularization argument to analyze the regularity of the $\dbar$-Neumann operator and harmonic forms. Let $\rho$ be a defining function for $\Omega$ normalized so that $|d\rho|=1$.  In real coordinates for $\mathbb{R}^{2n}$, we define the tangential gradient
\begin{multline*}
  (\nabla_T u,\nabla_T v)_\phi=\\
  \sum_{j=1}^{2n}\left(\left(\frac{\partial}{\partial x_j}-\frac{\partial\rho}{\partial x_j}\sum_{k=1}^{2n}\frac{\partial\rho}{\partial x_k}\frac{\partial}{\partial x_k}\right)u,\left(\frac{\partial}{\partial x_j}-\frac{\partial\rho}{\partial x_j}\sum_{\ell=1}^{2n}\frac{\partial\rho}{\partial x_\ell}\frac{\partial}{\partial x_\ell}\right)v\right)_\phi,
\end{multline*}
for $u,v\in L^{2,1}(\Omega,\phi)$ with the corresponding tangential Laplacian
\[
  \Delta_{T,\phi} u=\sum_{j=1}^{2n}\left(\frac{\partial}{\partial x_j}-\frac{\partial\rho}{\partial x_j}\sum_{\ell=1}^{2n}\frac{\partial\rho}{\partial x_\ell}\frac{\partial}{\partial x_\ell}\right)^{*,\phi}\left(\frac{\partial}{\partial x_j}-\frac{\partial\rho}{\partial x_j}\sum_{k=1}^{2n}\frac{\partial\rho}{\partial x_k}\frac{\partial}{\partial x_k}\right)u
\]
for $u\in L^{2,2}(\Omega,\phi)$.  Define the quadratic forms
\[
Q_\phi^{\delta,\nu} (u,v) = (\dbar u,\dbar v)_\phi + (\dbars_\phi u,\dbars_\phi v)_\phi + \delta(\nabla_T u, \nabla_T v)_\phi  + \nu(u,v)_\phi
\]
for $u,v \in\Dom(\dbars_\phi) \cap L^{2,1}_{0,q}(\Om,\phi)$ when $\delta>0$, or $u,v \in\Dom(\dbar)\cap\Dom(\dbars_\phi)$ when $\delta=0$.  When $\delta>0$, we may use, for example, Lemma 2.2 in \cite{Str10} to show that there exists a constant $c_{\delta,\nu}>0$ such that
\begin{equation}\label{eqn:u W^1 bdd by Q delta nu}
\|u\|_{L^{2,1}(\Om,\phi)}^2\leq c_{\delta,\nu} Q_\phi^{\delta,\nu}(u,u)\quad \text{holds for all
$u\in L^{2,1}_{0,q}(\Om,\phi)\cap\Dom(\dbars_\phi)$.}
\end{equation}
When $\delta$ or $\nu$ is equal to $0$, we omit the corresponding superscript, so, for example, $Q^{0,0}(u,v) = Q(u,v)$.  We may use standard techniques to construct the corresponding Laplacians
\[
\Box_\phi^{\delta,\nu} =\Box_\phi+\delta\Delta_{T,\phi} + \nu I
\]
with appropriate domains $\Dom(\Box_\phi^{\delta,\nu})$ (see, for example, Section 2.8 in \cite{Str10} when $\delta=0$ or Section 3.3 in \cite{Str10} when $\delta>0$).
Since \eqref{eqn:u W^1 bdd by Q delta nu} implies that $\Box_\phi^{\delta,\nu}$ is elliptic when $\delta>0$, we have $u\in L^{2,2}_{0,q}(\Omega,\phi)$ whenever $\Box_\phi^{\delta,\nu} u\in L^2(\Omega,\phi)$, and hence $\Dom(\Box_\phi^{\delta,\nu})\subset L^{2,2}_{0,q}(\Omega,\phi)$.  Since the term with the coefficient of $\delta$ involves only tangential derivatives, we have
\[
  \Dom(\Box_\phi^{\delta,\nu})=L^{2,2}_{0,q}(\Omega,\phi)\cap\Dom(\Box_\phi).
\]

\begin{thm}\label{thm:regularization}
Let $\Om\subset\C^n$ be a bounded domain and $1\leq q\leq n$.
Assume that
\begin{enumerate}[1.]
\item There is a constant $c>0$ so that
for any $u\in \Dom(\dbar)\cap \Dom(\dbars_\phi)\cap L^2_{0,q}(\Om,\phi)$, the following $L^2$-basic estimate holds:
\begin{equation}\label{eqn:L2 appendix}
\|u\|_{L^2(\Om,\phi)}^2\leq c\left(\|\dbar u\|_{L^2(\Om,\phi)}^2+\|\dbars_\phi u\|_{L^2(\Om,\phi)}^2+\|u\|_{L^{2,-1}(\Om,\phi)}^2\right)
\end{equation}

\item For some fixed $s_0 \in \N$ and all $0 \leq s \leq s_0$, there exists a constant $c_s>0$ so that
\begin{equation}\label{eqn:a priori for s}
\|u\|_{L^{2,s}(\Om,\phi)}^2\leq c_s\left(\|\Box_\phi^\delta u\|_{L^{2,s}(\Om,\phi)}^2+\|u\|_{L^2(\Om,\phi)}^2\right).
\end{equation}
for any $u\in L^{2,s}_{0,q}(\Om, \phi)$ so that $u\in \Dom(\Box_\phi)$ and  $\Box_\phi^\delta u \in L^{2,s}_{0,q}(\Om, \phi)$.
\end{enumerate}
Then $\H_{0,q}(\Om,\phi)\subset L^{2,s_0}_{0,q}(\Om,\phi)$ and the $\dbar$-Neumann operator $N^q_\phi$ is exactly regular on $L^{2,s}_{0,q}(\Om,\phi)$
for $0 \leq s \leq s_0$.

\end{thm}

\begin{proof} Suppose $s \leq s_0$ is a positive integer. Then
from \eqref{eqn:a priori for s}, there is an $\nu_{s_0}$ such that for any $\nu<\nu_{s_0}$ the following estimate
\begin{equation}\label{eqn:Box delta nu}
\|u\|_{L^{2,s}(\Om,\phi)}^2\leq 2c_s\left(\|\Box_\phi^{\delta,\nu} u\|_{L^{2,s}(\Om,\phi)}^2+\|u\|_{L^2(\Om,\phi)}^2\right)
\end{equation}
holds for any $u\in L^{2,s}_{0,q}(\Om,\phi)\cap \Dom(\Box_\phi)$ satisfying $\Box_\phi^{\delta,\nu} u\in L^{2,s}_{0,q}(\Om,\phi)$.
By construction, for any $\nu>0$ and $\delta\geq 0$ we also have
\begin{equation}\label{eqn:L2 Q delta nu}
\|u\|_{L^2(\Om,\phi)}^2\leq \frac{1}{\nu}Q_\phi^{\delta,\nu}(u,u)
\end{equation}
for all $u\in L^{2,1}_{0,q}(\Om,\phi) \cap \Dom(\dbars_\phi)$.

Consequently,
$\Box_\phi^{\delta,\nu}$ has closed range and a trivial kernel. This means $\Box_\phi^{\delta,\nu}$ has a continuous inverse on
$L^2_{0,q}(\Om,\phi)$ that we denote by $N_\phi^{\delta,\nu;q}$.
Also, for each $\nu> 0$, the inverse $N_\phi^{\delta,\nu;q}$ satisfies
\begin{equation}\label{eqn:L2 G_q delta nu}
\|N_\phi^{\delta,\nu;q}\alpha\|_{L^2(\Om,\phi)}^{2}\leq \frac{1}{\nu}\|\alpha\|_{L^2(\Om,\phi)}^{2} \quad \text{for all $\alpha\in L^2_{0,q}(\Om,\phi)$}.
\end{equation}

\textbf{Step 1:}  We will first show that if $\alpha \in  L^{2,s}_{0,q}(\Om,\phi)$ then $N_\phi^{0,\nu;q} \alpha \in  L^{2,s}_{0,q}(\Om,\phi)\cap \Dom(\Box_\phi)$.
By \eqref{eqn:u W^1 bdd by Q delta nu}, it follows that $\Box_\phi^{\delta,\nu}$ is elliptic which means that
if $\alpha\in L^{2,s}_{0,q}(\Om,\phi)$,
then $N_\phi^{\delta,\nu;q} \alpha\in L^{2,s+2}_{0,q}(\Om,\phi)\cap\Dom(\Box_\phi^{\delta,\nu})$.
Moreover, $L^{2,s+2}_{0,q}(\Om,\phi)\cap\Dom(\Box_\phi^{\delta,\nu}) \subset\Dom(\Box_\phi)$.
We can therefore use (\ref{eqn:Box delta nu}) with $u=N_\phi^{\delta,\nu;q}\alpha$ and estimate
\begin{multline}
\|N_\phi^{\delta,\nu;q}\alpha\|^2_{L^{2,s}(\Om,\phi)}
\leq 2 c_s\left( \|\Box_\phi^{\delta,\nu}N_\phi^{\delta,\nu;q}\alpha\|^2_{L^{2,s}(\Om,\phi)}
+\|N_\phi^{\delta,\nu;q} \alpha\|_{L^2(\Om,\phi)}^2\right) \\
= 2c_s\left( \|\alpha\|^2_{L^{2,s}(\Om,\phi)}+\|N_\phi^{\delta,\nu;q} \alpha\|_{L^2(\Om,\phi)}^2\right)
\leq 2c_s \|\alpha\|^2_{L^{2,s}(\Om,\phi)}+c_{s,\nu}\|\alpha\|_{L^2(\Om,\phi)}^2 \label{eqn:N delta nu s in terms of s and 0}
\end{multline}
for any positive integer $s \leq s_0$. The equality in \eqref{eqn:N delta nu s in terms of s and 0}
follows from the identity $\Box_\phi^{\delta,\nu}N_\phi^{\delta,\nu;q}=I$ and the fact that $\ker(\Box_\phi^{\delta,\nu} )=\{0\}$. The (second)
inequality follows by \eqref{eqn:L2 G_q delta nu} and the independence of the constants $c_s, c_{s,\nu}$ on $\delta>0$.

Thus, $\|N_\phi^{\delta,\nu;q} \alpha\|_{L^{2,s_0}(\Om,\phi)}$ is uniformly bounded in $\delta>0$.  Therefore, there exists a sequence $\delta_k\searrow 0$ such that $N^{\delta_{k},\nu;q}_\phi \alpha\to  u_{\nu}$ weakly in $L^{2,s_0}_{0,q}(\Om,\phi)$.  For any integer $0\leq s\leq s_0$, if $f\in L^{2,s}_{0,q}(\Om,\phi)$, then by the Riesz Representation Theorem there exists $\tilde f\in L^{2,s_0}_{0,q}(\Om,\phi)$ such that $(f,g)_{L^{2,s}(\Om,\phi)}=(\tilde f,g)_{L^{2,s_0}(\Om,\phi)}$ for all $g\in L^{2,s_0}(\Om,\phi)$, and hence $N^{\delta_{k},\nu;q}_\phi \alpha\to  u_{\nu}$ weakly in $L^{2,s}_{0,q}(\Om,\phi)$ for all integers $0\leq s\leq s_0$.  Thus, it follows that $u_\nu \in L^{2,s}_{0,q}(\Om,\phi)$, $0 \leq s \leq s_0$. Additionally,
$N^{\delta_k,\nu;q}_\phi \alpha\to  u_\nu$ weakly in the  $Q_\phi^{0,\nu}(\cdot,\cdot)^{1/2}$-norm. This
means that if $v\in L^{2,2}_{0,q}(\Om)\cap\Dom(\dbars_\phi)$, then
\[
\lim_{k\to\infty}Q_\phi^{0,\nu}(N^{\delta_k,\nu;q}_\phi\alpha,v)  =Q_\phi^{0,\nu}(u_\nu,v).
\]
On the other hand,
\begin{multline*}
Q_\phi^{0,\nu}(N^{0,\nu;q}_\phi\alpha,v)  =(\alpha,v)=
Q_\phi^{\delta,\nu}(N^{\delta,\nu;q}_\phi \alpha,v)\\
=Q_\phi^{0,\nu}(N^{\delta,\nu;q}_\phi \alpha,v) +\delta(\nabla_T N^{\delta,\nu;q}_\phi \alpha,\nabla_T v)
\end{multline*}
for all $v\in L^{2,2}_{0,q}(\Om,\phi)\cap\Dom(\dbars_\phi)$. It follows that
\begin{multline*}
\big|Q_\phi^{0,\nu}\big((N^{\delta,\nu;q}_\phi -N^{0,\nu;q}_\phi)\alpha,v\big)\big|
= \delta\big|(\nabla_T N^{\delta,\nu;q}_\phi \alpha,\nabla_T v)\big| \\
\leq \delta \|N^{\delta,\nu;q}_\phi\alpha\|_{L^2(\Om,\phi)}
\|v\|_{L^{2,2}(\Om,\phi)}\leq \delta c_\nu\|\alpha\|_{L^2(\Om,\phi)} \|v\|_{L^{2,2}(\Om,\phi)}
\end{multline*}
where we again used the inequality  $\|N^{\delta,\nu;q}_\phi \alpha\|_{L^2(\Om,\phi)}\leq c_\nu\|\alpha\|_{L^2(\Om,\phi)}$ uniformly in $\delta\geq 0$.
Thus,
\[
\lim_{\delta\to 0} Q_\phi^{0,\nu}\big((N^{\delta,\nu;q}_\phi -N^{0,\nu;q}_\phi)\alpha,v\big) =0.
\]
Since $\ker \Box_\phi^{\delta,\nu}=\{0\}$, it follows that $N^{0,\nu;q}_\phi\alpha =  u_\nu$ and hence $N^{0,\nu;q}_\phi\alpha \in L^{2,s}_{0,q}(\Om,\phi)$ for all integers $0\leq s\leq s_0$.
Moreover, we may apply \eqref{eqn:Box delta nu} with $u=N^{0,\nu;q}_\phi\alpha \in L^{2,s}_{0,q}(\Om,\phi)\cap \Dom(\Box_\phi)$ and $\delta=0$ and observe
\begin{equation}\label{est:G0nu}
\|N^{0,\nu;q}_\phi\alpha\|^2_{L^{2,s}(\Om,\phi)}\leq 2c_s\left( \|\alpha\|^2_{L^{2,s}(\Om,\phi)}+\|N^{0,\nu;q}_\phi\alpha\|_{L^2(\Om,\phi)}^2\right),
\end{equation}
holds for all $\alpha\in L^{2,s}_{0,q}(\Om,\phi)$.

\textbf{Step 2:} We next show that $\H_{0,q}(\Om,\phi) \subset L^{2,s}_{0,q}(\Om,\phi)$ for $0 \leq s \leq s_0$.

By Lemma~\ref{lem:harm forms, L^2 ests}, the space of $L^2$-harmonic forms $\H_{0,q}(\Om,\phi)$ is finite dimensional.
Let $\theta_1,\cdots,\theta_N\in \H_{0,q}(\Om,\phi)$ be an orthonormal basis, and set $\theta_0=0$.
We will prove $\theta_j\in L^{2,s_0}_{0,q}(\Om,\phi)(\Om)$ for all $j$ by induction.
Certainly $\theta_0 \in L^{2,s_0}_{0,q}(\Om,\phi)(\Om)$.
Assume now that for some $0\leq k<N$,
$\theta_j\in L^{2,s_0}_{0,q}(\Om,\phi)(\Om)$ for all $0 \leq j \leq k$.
We will construct $\theta\in \H_{0,q}(\Om,\phi)\cap L^{2,s_0}_{0,q}(\Om,\phi)(\Om)$ with $\|\theta\|_{L^2(\Om,\phi)}=1$
and $(\theta,\theta_j)_{L^2(\Om,\phi)}=0$ for $j\leq k$. If we replace $\theta_{k+1}$ with $\theta$, we may proceed by induction to obtain a basis of
$\H_{0,q}(\Om,\phi)$ which is contained in $L^{2,s_0}_{0,q}(\Om,\phi)$.

Let $\alpha\in L^{2,s_0}_{0,q}(\Om,\phi)$ be a form
such that $\alpha$ is orthogonal (in $L^2_{0,q}(\Om,\phi)$) to $\theta_j$ for $j\leq k$ but not to $\theta_{k+1}$.  This can be obtained, for example, by regularizing $\theta_{k+1}$ and projecting onto the orthogonal complement of the span of $\{\theta_1,\ldots,\theta_k\}$.
Then, for $\nu>0$, $N^{0,\nu;q}_\phi\alpha\in L^{2,s_0}_{0,q}(\Om,\phi)\cap \Dom(\Box_\phi)$ and satisfies \eqref{est:G0nu}.
We claim that $\{\|N^{0,\nu;q}_\phi\alpha\|_{L^2(\Om,\phi)} : 0 < \nu < 1 \}$ is unbounded.
If it were bounded then by \eqref{est:G0nu} we could find a subsequence converging (weakly) to a form $u\in L^{2,s_0}_{0,q}(\Om,\phi)\cap \Dom(\Box_\phi)$
satisfying
\[
Q_\phi(u,\psi)=(\alpha,\psi)
\]
for all $\psi\in \Dom(Q_\phi)$. By setting $\psi = \alpha$, we see that $u\neq 0$, and if
$\psi=\theta_{k+1}$, the left-hand side is
zero while the right-hand side is different from zero, a contradiction.
Thus the set $\{\|N_\phi^{0,\nu;q}\alpha\|_{L^2(\Om,\phi)}\}$ is unbounded and we can therefore find a subsequence
$\{\|N_\phi^{0,\nu_m;q}\alpha\|_{L^2(\Om,\phi)}\}$
such that $\lim_{m\to \infty}\|N_\phi^{0,\nu_m;q}\alpha\|_{L^2(\Om,\phi)}=\infty$ and $\nu_m\to 0$.
Set $w_m=\frac{N_\phi^{0,\nu_m;q}\alpha}{\|N_\phi^{0,\nu_m;q}\alpha\|}_{L^2(\Om,\phi)}$.
Then $w_m\in L^{2,s_0}_{0,q}(\Om,\phi)\cap \Dom(\Box_\phi)$, $\|w_m\|_{L^2(\Om,\phi)}=1$, and by \eqref{est:G0nu}
\[
\|w_m\|_{L^{2,s_0}(\Om,\phi)}\leq c_{s_0} \left(\frac{\|\alpha\|_{L^{2,s_0}(\Om,\phi)}}{\|N_\phi^{0,\nu_m;q}\alpha\|_{L^2(\Om,\phi)}}+1\right).
\]
Thus, there is a subsequence $\{w_{m_j}\}$ converging weakly to $\theta\in L^{2,s_0}_{0,q}(\Om,\phi)$. The compact inclusion $L^{2,s_0}_{0,q}(\Om,\phi) \hookrightarrow L^2_{0,q}(\Om,\phi)$ forces
norm convergence of $w_{m_j}$ to $\theta$ in $L^2_{0,q}(\Om,\phi)$.
Thus, $\|\theta\|_{L^2(\Om,\phi)}=1$.
To see that $\theta\in \H_{0,q}(\Om,\phi)$, we use the inequality
\begin{multline*}
Q_\phi(w_{m_j},w_{m_j})\leq Q_\phi^{0,\nu_{m_j}}(w_{m_j},w_{m_j})\\
=\frac{1}{\|N_\phi^{0,\nu_{m_j};q}\alpha\|_{L^2(\Om,\phi)}}(\alpha, w_{m_j})
\leq \frac{\|\alpha\|_{L^2(\Om,\phi)}}{\|N_\phi^{0,\nu_{m_j};q}\alpha\|_{L^2(\Om,\phi)}}.
\end{multline*}
Indeed, since $\{w_{m_j}\}$ converges weakly to $\theta$ in $L^{2,1}_{0,q}(\Om,\phi)$, we have
\begin{multline*}
Q_\phi(\theta,\theta)=\lim_{j\to\infty}\lim_{k\rightarrow\infty} Q_\phi(w_{m_j},w_{m_k})\\
\leq\lim_{j\to\infty}\lim_{k\rightarrow\infty} Q_\phi(w_{m_j},w_{m_j})^{1/2}Q_\phi(w_{m_k},w_{m_k})^{1/2}\\
=\left(\lim_{j\to\infty}Q_\phi(w_{m_j},w_{m_j})^{1/2}\right)^2=0.
\end{multline*}
Hence $\theta\in \H_{0,q}(\Om,\phi)$.

Finally, to prove $(\theta,\theta_j)=0$ for $j\leq k$, for any $0\leq j\leq k$ we have
\[
\nu_m(w_m,\theta_j)=Q_\phi^{\nu_m}(w_m,\theta_j)=\frac{1}{\|N_\phi^{0,\nu_{m};q}\alpha\|_{L^2(\Om,\phi)}}(\alpha,\theta_j)=0.
\]
This means $w_m$ is orthogonal to $\theta_k$ for $j=1,\dots,k$ and so $\theta$ is as well.
Therefore, $\H_{0,q}(\Om,\phi)\subset L^{2,s_0}_{0,q}(\Om,\phi)$.

\textbf{Step 3}: Finally, we show that $N^q_\phi:=N^{0,0;q}_\phi$ is exactly regular on $L^{2,s}_{0,q}(\Om,\phi)$ for $0 \leq s \leq s_0$.
We start this step by combining Lemma~\ref{lem:harm forms, L^2 ests} and \eqref{eqn:L2 appendix}. In particular, for any
$u \in \Dom(\dbar)\cap\Dom(\dbars_\phi) \cap \H^\perp_{(0,q)}(\Om,\phi)$
\[
\|u\|_{L^2(\Om,\phi)}^2\leq c\left(\|\dbar u\|_{L^2(\Om,\phi)}^2+\|\dbars_\phi u\|_{L^2(\Om,\phi)}^2\right)= c Q_\phi(u,u)
\]
and hence
\begin{equation}\label{eqn:uniformly}
\|u\|_{L^2(\Om,\phi)}^2\leq c \left( Q_\phi(u,u)+\nu \|u\|_{L^2(\Om,\phi)}^2 \right)=c Q_\phi^{0,\nu}(u,u)
\end{equation}
where $c$ is independent of $\nu$.

By the definition of $Q_\phi^{0,\nu}$ and $N_\phi^{0,\nu;q}$, we have
\[
Q_\phi(N_\phi^{0,\nu;q}\alpha, \psi)+\nu(N_\phi^{0,\nu;q}\alpha, \psi)=Q_\phi^{0,\nu}(N_\phi^{0,\nu;q}\alpha, \psi)=(\alpha,\psi)
\]
for any $\alpha,\psi\in L^2_{0,q}(\Om,\phi)$ with $\psi\in\Dom(\dbar)\cap\Dom(\dbars_\phi)$. Thus, if $\alpha\perp \H_{0,q}(\Om,\phi)$ and $\psi\in \H_{0,q}(\Om,\phi)$,
then it follows that $N_\phi^{0,\nu;q}\alpha\perp  \H_{0,q}(\Om,\phi)$ since a consequence of
the harmonicity of $\psi$ is that $Q_\phi(f,\psi) =0$ for any $f \in \Dom \dbar \cap \Dom\dbars_\phi$.
Thus, if $u=N_\phi^{0,\nu;q}\alpha$ and $\alpha\perp\H_{0,q}(\Om,\phi)$, then the uniformity of \eqref{eqn:uniformly} (in $\nu >0$) implies
\[
\|N_\phi^{0,\nu;q}\alpha\|_{L^2(\Om,\phi)}\leq c\|\alpha\|_{L^2(\Om,\phi)}.
\]
Combining this uniform $L^2$ estimate with \eqref{est:G0nu} yields the uniform (in $\nu>0$) $L^{2,s}$-estimate
\begin{equation}\label{est:G0nu phi}
\|N_\phi^{0,\nu;q}\alpha\|^2_{L^{2,s}(\Om,\phi)}\leq c_s\left( \|\alpha\|^2_{L^{2,s}(\Om,\phi)}+\|\alpha\|_{L^2(\Om,\phi)}^2\right)
\leq c_s\|\alpha\|^2_{L^{2,s}(\Om,\phi)},\end{equation}
for any $\alpha\in L^{2,s}_{0,q}(\Om,\phi)\cap \H^{\perp}_{0,q}(\Om)$.

Now we use argument of Step 1  to  send $\nu\to 0$ and establish that $N^q_\phi\alpha\in L^{2,s}_{0,q}(\Om,\phi)\cap \H^{\perp}_{0,q}(\Om,\phi)$
and \eqref{est:G0nu phi} holds for $\nu=0$. For $\alpha\in L^{2,s}_{0,q}(\Om,\phi) $, we decompose
$\alpha=(I-H^q_\phi)\alpha+H^q_\phi\alpha$ (recall that $H^q_\phi$ is the orthogonal projection $L^2_{0,q}(\Om,\phi)$
onto $\H_{0,q}(\Om,\phi)$).
Since $\alpha\in L^{2,s}_{0,q}(\Om,\phi)$,
it follows from Step 2 that $(I-H^q_\phi)\alpha\in L^{2,s}_{0,q}(\Om,\phi)$,
and by using \eqref{est:G0nu phi} for $\nu=0$,
we may conclude that
\begin{equation}\label{eqn:N_q est in W^s}
\|N^q_\phi\alpha\|^2_{L^{2,s}(\Om,\phi)}\leq c_s \|(I-H^q_\phi) \alpha\|^2_{L^{2,s}(\Om,\phi)}\leq c_s\|\alpha\|^2_{L^{2,s}(\Om,\phi)} ,
\end{equation}
for all $\alpha\in L^{2,s}_{0,q}(\Om,\phi)$. That means $N^q_\phi$ is exactly regular in $L^{2,s}_{0,q}(\Om,\phi)$, $0 \leq s \leq s_0$.
\end{proof}

We turn to showing that a family of closed range estimates will suffice to satisfy the hypotheses of Theorem \ref{thm:regularization} for sufficiently large $t$.
\begin{prop}\label{prop:SCRE imply regularity at s}
Let $\Om\subset\C^n$ be a smooth domain which admits the family of strong closed range estimates \eqref{eq:temp_strong_closed_range_2}
for some smooth function $\varphi$. Then for every $k\geq 1$ there exists $T_k$ so that if $t \geq t_k$, then
the hypotheses of Theorem \ref{thm:regularization} hold for $s_0=k$.
\end{prop}

\begin{proof}
Suppose that $X^k$ is a real order $k$ differential operator that is tangential on $\bd\Om$.  We define the action of $X^k$ on differential forms by locally writing each form in a special boundary chart (see 2.2 in \cite{Str10}, for example) and applying $X^k$ to the coefficients of the form in this chart.  Hence, $X^k$ will preserve the domain of $\dbar^*_{t\varphi}$.

We first note that (5.3) in \cite{Str10} holds in our case: for any $k\geq 1$, if $X^k$ is a real order $k$ differential operator that is tangential on $\bd\Om$, then we have
\begin{multline}
\label{eq:Q_form_estimate_k}
  \|\dbar X^k f\|_{L^2(\Om,t\varphi)}^2 + \|\dbars_{t\varphi} X^k f\|_{L^2(\Om,t\varphi)}^2
+ \delta \|\nabla_T X^k f\|_{L^2(\Om,t\varphi)}^2\\
\leq C\left(\norm{\Box_{t\varphi}^\delta f}^2_{L^{2,k}(\Om,t\varphi)}+\norm{f}^2_{L^{2,k}(\Om,t\varphi)}\right)+C_t\norm{f}^2_{L^2(\Om,t\varphi)},
\end{multline}
for all $f\in W^{2k+1}_{0,q}(\Om,t\varphi)\cap\dom(\Box_{t\varphi})$, where $C>0$ is a constant that is independent of $f$ and $t$, and $C_t>0$ is a constant that is only independent of $f$.  If we make the substitution $f=N^\delta_{t\varphi} u$, then the only difference between \eqref{eq:Q_form_estimate_k} and (5.3) in \cite{Str10} is the final term, which would be $C_t\norm{\Box^\delta_{t\varphi}f}^2_{L^2(\Om,t\varphi)}$ in our notation.  This relies on the estimate $\norm{f}^2_{L^2(\Om,t\varphi)}\leq C\norm{\Box^\delta_{t\varphi}f}^2_{L^2(\Om,t\varphi)}$, which is true in the pseudoconvex case studied by Straube, but not necessarily in our case.  Since our domain is not necessarily pseudoconvex, we also note that \eqref{eq:Q_form_estimate_k} may fail when $k=0$.

If $f \in W^{2k+1}_{0,q}(\Om,t\varphi)\cap\Dom \dbars_{t\varphi}$, then
$X^k f\in\Dom(\dbars_{t\varphi})$ so that \eqref{eq:temp_strong_closed_range_2} holds. Consequently,
\begin{multline}\label{eqn:X^k f}
tC_q\|X^k f\|_{L^2(\Om,t\varphi)}^2
\leq\\ \|\dbar X^k f\|_{L^2(\Om,t\varphi)}^2 + \|\dbars_{t\varphi} X^k f\|_{L^2(\Om,t\varphi)}^2
+ \delta \|\nabla_T X^k f\|_{L^2(\Om,t\varphi)}^2 + C_t \|f\|_{W^{k-1}(\Om,t\varphi)}^2.
\end{multline}
Plugging \eqref{eq:Q_form_estimate_k} into \eqref{eqn:X^k f}, we see that for any $f \in W^{2k+1}_{0,q}(\Om,t\varphi)\cap\Dom \Box_{t\varphi}$
\begin{multline*}
\|X^k f\|_{L^2(\Om,t\varphi)}^2
\leq \frac{C}{t}\big(\|\Box^\delta_{t\varphi} f\|_{L^{2,k}(\Om,t\varphi)}^2
+\|f\|_{L^{2,k}(\Om,t\varphi)}^2 \big)\\
  + C_t\left(\|f\|_{W^{k-1}(\Om,t\varphi)}^2+\norm{f}^2_{L^2(\Om,t\varphi)}\right).
\end{multline*}
As noted in the proof of (5.3) in \cite{Str10}, we may use Sobolev interpolation to estimate $\|f\|_{W^{k-1}(\Om,t\varphi)}^2\leq \eps\|f\|_{W^{k}(\Om,t\varphi)}^2+C_\eps\|f\|_{L^2(\Om,t\varphi)}^2$, so we have
\[
\|X^k f\|_{L^2(\Om,t\varphi)}^2
\leq \frac{C}{t}\big(\|\Box^\delta_{t\varphi} f\|_{L^{2,k}(\Om,t\varphi)}^2
+\|f\|_{L^{2,k}(\Om,t\varphi)}^2 \big)
  + C_t\norm{f}^2_{L^2(\Om,t\varphi)}.
\]
Using, for example, Lemma 2.2 in \cite{Str10}, we see that for forms $f\in \Dom(\dbars_{t\varphi})$, normal derivatives of $f$ are controlled by $\dbar f$, $\dbars_{t\varphi}f$, tangential derivatives of $f$, and $f$ itself.  For higher order normal derivatives, we may use, for example, (3.42) in \cite{Str10} to reduce the order and proceed by induction on the number of normal derivatives.  It follows that for $f \in W^{2k+1}(\Om,t\varphi) \cap \Dom(\dbars_{t\varphi})$
\[
\|f\|_{L^{2,k}(\Om,t\varphi)}^2
\leq \frac{C}{t}\big(\|\Box^\delta_{t\varphi} f\|_{L^{2,k}(\Om,t\varphi)}^2
+\|f\|_{L^{2,k}(\Om,t\varphi)}^2 \big)
  + C_t\norm{f}^2_{L^2(\Om,t\varphi)}.
\]
Thus, by choosing $t$ large enough,
\[
\|f\|_{L^{2,k}(\Om,t\varphi)}^2
\leq \frac{C}{t}\big(\|\Box^\delta_{t\varphi} f\|_{L^{2,k}(\Om,t\varphi)}^2  \big)  + C_t \|f\|_{L^2(\Om,t\varphi)}^2.
\]
While this inequality holds for forms $f \in L^{2,2k+1}_{0,q}(\Om,t\varphi)\cap\Dom(\dbars_{t\varphi})\cap\Dom(\Box^\delta)$, this space of forms is dense in
$L^{2,k}_{0,q}(\Om,t\varphi)\cap \Dom(\Box_{t\varphi})$ for which $\Box^\delta_{t\varphi}f \in W^k(\Om,t\varphi)$.  Thus, the proof of the proposition
is complete.
\end{proof}


\begin{proof}[Proof of Theorem \ref{thm:regularity of B and sol'n ops}]
We have already proven that $N^q_{t\varphi} : L^{2,s}_{0,q}(\Om,t\varphi) \to L^{2,s}_{0,q}(\Om,t\varphi)$ in Theorem \ref{thm:regularization} and Proposition \ref{prop:SCRE imply regularity at s}, so (i) is proven. Additionally, in Theorem \ref{thm:regularization} we proved
that $H^q_{t\varphi}$ is continuous on the same range of $s$ (this requires the Closed Graph Theorem and the fact that $H^q_{t\varphi}$ is a closed operator).  With this in mind, estimates for $\dbars_{t\varphi} N^q_{t\varphi}$, $\dbar N^q_{t\varphi}$, $\dbars_{t\varphi}\dbar N^q_{t\varphi}$, and $\dbar\dbars_{t\varphi} N^q_{t\varphi}$ will follow from the proof of Theorem 5.1 in \cite{Str10}.  The key difference is that we will use the identity $\Box_{t\varphi}N^q_{t\varphi}=I-H^q_{t\varphi}$, but we have already proven estimates for $H^q_{t\varphi}$.

For operators such as $N^q_{t\varphi} \dbar$, observe its adjoint $(N^q_{t\varphi} \dbar)^*_{t\varphi} = \dbars_{t\varphi} N^q_{t\varphi}$
is a bounded operator from $L^2_{0,q}(\Om,t\varphi) \to L^2_{0,q-1}(\Om,t\varphi)$. Consequently, since $N^q_{t\varphi} \dbar$ agrees
with $(\dbars_{t\varphi}N^q_{t\varphi})^*$ on $\Dom(\dbar)$, we simply extend $N^q_{t\varphi}\dbar$ to be the bounded
operator from $L^2_{0,q-1}(\Om,t\varphi) \to L^2_{0,q}(\Om,t\varphi)$ that agrees with
$(\dbars_{t\varphi} N^q_{t\varphi})^*$. When $N^{q-1}_{t\varphi}$ exists, this extension satisfies $N^q_{t\varphi} \dbar=\dbar N^{q-1}_{t\varphi}$.  Similarly, we can extend $N^q_{t\varphi} \dbars_{t\varphi}$ to be a well-defined operator on the appropriate weighted $L^2$-spaces.  Note that on the space $W^1_{0,q+1}(\Om,t\varphi)$, we have $N^q_{t\varphi} \dbars_{t\varphi}
= N^q_{t\varphi}\vartheta_{t\varphi}$.

To estimate $N^q_{t\varphi}\dbar$ and $\dbars_{t\varphi}N^q_{t\varphi}\dbar$, we observe that the proof of Theorem 5.1 in \cite{Str10} concludes by proving estimates for $N^1_{t\varphi}\dbar$ and $\dbars_{t\varphi}N^1_{t\varphi}\dbar$, and this proof is easily generalized to the $q>1$ case.  The same proof can be adapted to estimate $N^q_{t\varphi}\dbars_{t\varphi}$ and $\dbar N^q_{t\varphi}\dbars_{t\varphi}$.

\end{proof}

Recall the following well-known fact (cf. Theorem 3.19 in \cite{Koh73}, see also \cite{RaSt08}).
\begin{lem}\label{lem:dim of harm forms}
Let $\Om \subset\C^n$ be a domain satisfying \eqref{eq:temp_strong_closed_range_2} for some $1\leq q\leq n$ and $\varphi:\bar\Om\to\C$ a (smooth) bounded function. Then for all $t\in\R$,
$\dim_\C \H_{0,q}(\Om,t\varphi) = \dim_\C \H_{0,q}(\Om)$.
\end{lem}

\begin{proof}
This follows immediately from the observation that $\mathcal{H}_{0,q}(\Omega,t\varphi)$ is the orthogonal complement of $\Ran\dbar$ in $\ker\dbar$.  Since $\Ran\dbar$ and $\ker\dbar$ are independent of
the weight $t\varphi$, the orthogonal complement of $\Ran\dbar$ in $\ker\dbar$ has the same dimension, whether measured in the weighted
or unweighted spaces.
\end{proof}

\begin{proof}[Proof of Corollary \ref{cor: no harmonic forms, solvability in C^infty}]
Let $t$ be chosen sufficiently large so that we have estimates for $I-\dbars_{t\varphi}\dbar N^q_{t\varphi}$ in $L^{2,m}_{0,q}(\Omega,t\varphi)$.  If $f\in L^{2,s}_{0,q}(\Om)$ is $\dbar$-closed, then $f\in L^{2,s}_{0,q}(\Omega,t\varphi)$ since the spaces are equivalent on bounded domains.
A $\dbar$-closed approximation in $L^{2,m}_{0,q}(\Omega,t\varphi)$ is produced as follows:
Let $\tilde f$ be an approximation in $L^{2,m}_{0,q}(\Omega,t\varphi)$. Then $\tilde f-\dbars_{t\varphi}\dbar N^q_{t\varphi}\tilde f\in L^{2,m}_{0,q}(\Omega,t\varphi)$ is a $\dbar$-closed
approximation of $f$ for $t$ sufficiently large and $\tilde f$ sufficiently close to $f$ in $L^{2,s}_{0,q}(\Om,t\varphi)$.  This will also be an approximation in $L^{2,m}_{0,q}(\Omega)$ since the norm on this space is equivalent to the norm on $L^{2,m}_{0,q}(\Omega)$ for fixed $t$ when $\Omega$ is bounded.

Smooth solvability will follow from the proof
of Theorem 6.1.1 in \cite{ChSh01} (see also \cite[Section 6.8]{HaRa11}).  It suffices to note that these proofs require Sobolev regularity for the weighted Bergman
projection $I-\dbars_{t\varphi}N^q_{t\varphi}\dbar$ and the weighted canonical solution operator $\dbars_{t\varphi}N^q_{t\varphi}$ (when $\tilde q=q-1$) or $I-\dbars_{t\varphi}\dbar N^{q}_{t\varphi}$ and $N^{q}_{t\varphi}\dbars_{t\varphi}$ (when $\tilde q=q$).
\end{proof}

\bibliographystyle{amsplain}
\bibliography{mybib3-12-19}
\end{document}